\documentclass[letterpaper,final]{siamltex}
\usepackage{amsfonts,amsmath,amssymb}
\usepackage{graphicx}
\usepackage[ruled,titlenotnumbered,nofillcomment]{algorithm2e}
\usepackage{multirow,stmaryrd}
\usepackage[usenames]{color}
\usepackage{epsfig}
\usepackage{epsf}
\usepackage{subfigure}

\newcommand{\beq}{\begin{equation}}
\newcommand{\eeq}{\end{equation}}
\newcommand{\beqs}{\begin{equation*}}
\newcommand{\eeqs}{\end{equation*}}

\newcommand{\bA}{{\bf A}}

\newcommand{\bC}{{\bf C}}
\newcommand{\bG}{{\bf G}}
\newcommand{\bP}{{\bf P}}
\newcommand{\bbT}{{\bf T}}
\newcommand{\bM}{{\bf M}}
\newcommand{\bK}{{\bf K}}
\newcommand{\bI}{{\bf I}}
\newcommand{\ba}{{\bf a}}
\newcommand{\bb}{{\bf b}}
\newcommand{\bx}{{\bf x}}
\newcommand{\by}{{\bf y}}

\newcommand{\bd}{{\bf d}}

\newcommand{\bff}{{\bf f}}
\newcommand{\bg}{{\bf g}}
\newcommand{\bu}{{\bf u}}
\newcommand{\be}{{\bf e}}
\newcommand{\bz}{{\bf z}}
\newcommand{\bp}{{\bf p}}
\newcommand{\br}{{\bf r}}

\newcommand{\mA}{\mathcal{A}}

\newcommand{\mN}{\mathcal{N}}

\newcommand{\mT}{\mathcal{T}}

\newcommand{\bw}{{\bf w}}

\newcommand{\Div}{\nabla\cdot}

\newcommand{\BC}{\begin{center}}
\newcommand{\EC}{\end{center}}
\newcommand{\bdm}{\begin{displaymath}}
\newcommand{\edm}{\end{displaymath}}
\newcommand{\CF}{{\cal F}}

\newcommand{\dia}[1]{\text{diag}(#1)}

\newcommand{\argmin}{\mathop{\mbox{argmin}}}

\newcommand{\barr}{\begin{array}}
\newcommand{\earr}{\end{array}}

\newcommand{\beqas}{\begin{eqnarray*}}
\newcommand{\eeqas}{\end{eqnarray*}}

\title{A Nonlinear GMRES Optimization Algorithm for Canonical Tensor Decomposition}

\author{
H. De Sterck\footnotemark[1] \footnotemark[4] 
}

\begin{document}
%
\maketitle
\renewcommand{\thefootnote}{\fnsymbol{footnote}}
\footnotetext[1]{Department of Applied Mathematics, University of Waterloo,
Waterloo, Ontario, Canada}
\footnotetext[4]{hdesterck@uwaterloo.ca}
\begin{abstract}
A new algorithm is presented for computing a canonical rank-$R$ tensor approximation that has minimal distance to a given tensor in the Frobenius norm, where the canonical rank-$R$ tensor consists of the sum of $R$ rank-one components.
Each iteration of the method consists of three steps. In the first step, a tentative new iterate is generated by a stand-alone one-step process, for which we use alternating least squares (ALS). In the second step, an accelerated iterate is generated by a nonlinear generalized minimal residual (GMRES) approach, recombining previous iterates in an optimal way, and essentially using the stand-alone one-step process as a preconditioner. In particular, the nonlinear extension of GMRES is used that was proposed by Washio and Oosterlee in {\bf [ETNA Vol.~15 (2003), pp.~165-185]} for nonlinear partial differential equation problems. In the third step, a line search is performed for globalization. The resulting nonlinear GMRES (N-GMRES) optimization algorithm is applied to dense and sparse tensor decomposition test problems. The numerical tests show that ALS accelerated by N-GMRES may significantly outperform both stand-alone ALS and a standard nonlinear conjugate gradient optimization method, especially when highly accurate stationary points are desired for difficult problems. 
The proposed N-GMRES optimization algorithm is based on general concepts and may be applied to other nonlinear optimization problems.
\end{abstract}
\pagestyle{myheadings}
\thispagestyle{plain}
\markboth{H. De Sterck
}{Nonlinear GMRES Optimization for Canonical Tensor Decomposition}

\section{Introduction}
In this paper, we present a new algorithm for computing a canonical rank-$R$ tensor approximation that has minimal distance to a given tensor in the Frobenius norm, where the canonical rank-$R$ tensor consists of the sum of $R$ rank-one components. 
As one of its components, the optimization 
algorithm uses a nonlinear version of the generalized minimal residual (GMRES) method that was originally proposed for iteratively solving systems of linear equations \cite{SaadGMRES}. More specifically, we use the nonlinear extension of GMRES that was developed by Washio and Oosterlee \cite{WashioNGMRES-ETNA} for nonlinear partial differential equation (PDE) systems, and apply it to the tensor optimization problem, with the purpose of efficiently driving the gradient of the objective function to zero in a process that uses iterate recombination. We apply this nonlinear GMRES acceleration to the alternating least squares (ALS) method, and combine it with a line search for globalization. 
We perform numerical tests to investigate the performance of the resulting nonlinear GMRES (N-GMRES) optimization algorithm for canonical tensor approximation.

$N$-way tensor $\mT \in \mathbb{R}^{I_1 \times \ldots \times I_N}$ is an $N$-dimensional array of size $I_1 \times \ldots \times I_N$ \cite{KoldaSIREV}.
The size of mode $n$ is $I_n$ ($n=1, \ldots, N$). Let $\mA_R \in \mathbb{R}^{I_1 \times \ldots \times I_N}$ be a canonical rank-$R$ tensor, given by
\begin{align}
\mA_R=\sum_{r=1}^{R} \, \ba_r^{(1)} \circ \ldots \circ \ba_r^{(N)}=[\![\bA^{(1)}, \ldots, \bA^{(N)} ]\!].
\label{eq:A}
\end{align}
Canonical tensor $\mA_R$ is a sum of $R$ rank-one tensors, with the $r$th rank-one tensor composed of the outer product of $N$ column vectors $\ba_r^{(n)} \in \mathbb{R}^{I_n}$, $n=1, \ldots, N$. For each mode $n=1,\ldots,N$, the $R$ vectors $\ba_r^{(n)}$, $r=1,\ldots,R$, form the columns of the mode-$n$ factor matrix $\bA^{(n)}$, and the double-bracket notation on the right of (\ref{eq:A}) is used to denote the canonical tensor by the factor matrices.

This paper concerns numerical algorithms for the following optimization problem:\\

\noindent
{\sc optimization problem I:}

\begin{center}
given tensor $\mT \in \mathbb{R}^{I_1 \times \ldots \times I_N}$, find rank-$R$\\
canonical tensor $\mA_R \in \mathbb{R}^{I_1 \times \ldots \times I_N}$ that minimizes\\
\end{center}
\begin{align}
f(\mA_R)=\frac{1}{2} \, \| \mT  - \mA_R\|_F^2.
\label{eq:fA}
\end{align}
Here, $\| . \|_F$ denotes the Frobenius norm of the $N$-dimensional array, which is the square root of the sum of the squares of all the array elements.

We now briefly recall some properties of canonical tensor decomposition, mainly referring to review article \cite{KoldaSIREV}, which also contains extensive references to original papers that the reader may consult for more detail.
The exact decomposition of a data tensor $\mT$ into a canonical tensor is often called a CP decomposition of the tensor, with the C standing for `CANDECOMP' and the `P' standing for `PARAFAC', after the names originally given to this decomposition in early papers on the subject \cite{CANDECOMP,PARAFAC}. The smallest number of rank-one components that generate $\mT$ as their sum is called the tensor rank of tensor $\mT$ \cite{KoldaSIREV}.
If, rather than an exact decomposition, $\mT$ is {\em approximately} decomposed into a low-rank canonical tensor $\mA_R$ of specified rank $R$ that is smaller than the rank of $\mT$, then the resulting tensor $\mA_R$ is called an approximate CP decomposition. CP decompositions are used for data analysis in many applications such as chemometrics, signal processing, neuroscience and web analysis. Many criteria exist for specifying the type of approximation that is sought for approximate CP decompositions (see, e.g., \cite{TomasiPARAFAC}). In this paper, we focus on the specific (and practically relevant) case of computing an approximate CP decomposition $\mA_R$ that minimizes the Frobenius distance between the data tensor and $\mA_R$, i.e., we seek to minimize objective function (\ref{eq:fA}). Contrary to the case of best rank-$R$ matrix approximation, the rank-one terms of the best rank-$R$ CP tensor approximation cannot be solved for sequentially but must be found simultaneously, since a best rank-$R$ CP approximation cannot be obtained by truncating a best rank-$S$ approximation with $S>R$ \cite{KoldaSIREV}.

It is well-known that Optimization Problem I is a non-convex optimization problem, and as such may exhibit multiple local minima. In the form given above, its local minima are not isolated, since there is a scaling indeterminacy in each of the rank-one terms: there is ample freedom to rescale the vectors $\ba_r^{(n)}$ in (\ref{eq:A}) without changing the rank-one product. In our approach, we deal with this by normalizing the $\ba_r^{(n)}$ in a specific way, to be explained below. There is also a permutation indeterminacy in the order of the rank-one terms, which we deal with by imposing a specific order, also to be explained below. Even when the scaling and permutation indeterminacies are removed, CP optimization may still exhibit multiple local minima for some problems, and depending on the initial guess, iterative methods for approximate CP decomposition may converge to different stationary points. It is also possible that the best rank-$R$ approximation does not exist, which may happen when some rank-one terms with opposite signs in the canonical tensor become unbounded in size as the canonical tensor approaches the target tensor, in a phenomenon called degeneracy \cite{KoldaSIREV}. On the other hand, exact CP decomposition has been shown to be unique up to scaling and permutation under mild conditions that relate the ranks of the factor matrices with the tensor rank \cite{KoldaSIREV}.

A multitude of algorithms and approaches exist for computing approximate CP decompositions, see, for example,
\cite{TomasiPARAFAC,KoldaSIREV,AcarCPOPT} and references therein. A standard numerical method for computing the CP approximation of Optimization Problem I is the ALS method, which was already proposed in early papers on CP decomposition \cite{CANDECOMP,PARAFAC}. ALS is simple to understand and implement, and often performs adequately, but its convergence can also be very slow, depending on the problem and the initial condition. Alternatives to ALS are described in, for example, \cite{TomasiPARAFAC,LathauwerSimul,LathauwerSchur}, see also the discussion in \cite{KoldaSIREV,AcarCPOPT}. Even though ALS is a simple algorithm, it has proven difficult over the years to develop
alternatives that significantly improve on it in a way that is robust over large classes of problems. As a results, ALS-type algorithms are still often considered as the `workhorse' algorithms of choice in practice.

In recent work on CP decomposition, Acar et al. \cite{AcarCPOPT} consider first-order optimization algorithms for Optimization Problem I. In particular, they employ the nonlinear conjugate gradient (N-CG) method and compare it with ALS and nonlinear least-squares algorithms. In order to formulate first-order optimization methods, they derive expressions for the gradient of objective function (\ref{eq:fA}) in a systematic way. At any (local) minimum of Optimization Problem I the following conditions hold:\\

\noindent
{\sc first-order optimality equations I:}
\begin{align}
\nabla f(\mA_R)=\bg(\mA_R)=0.
\label{eq:gA}
\end{align}
Detailed expressions for the gradient vector $\bg(\mA_R)$ will be given below.
Acar et al. \cite{AcarCPOPT} then apply a standard N-CG optimization method with line search and obtain convergence speeds that are competitive with ALS and sometimes better, depending on the problem.

In this paper we also follow a first-order optimization approach for Optimization Problem I, but we propose to use a different optimization method. In particular, we propose to use a nonlinear GMRES approach to recombine iterates provided by a standard iterative method for Optimization Problem I (we use ALS as the `GMRES preconditioner' in this paper), combined with line search for globalization. The resulting N-GMRES optimization algorithm will be shown numerically to accelerate convergence significantly in many cases, especially when stationary points need to be determined accurately for difficult problems, and we show this for two classes of dense and sparse test problems.
The N-GMRES optimization algorithm proposed is easy to implement as a wrapper around any iterative solution method for Optimization Problem I. It combines previous iterates and can thus, in that sense, also be considered as a generalization of line search methods for ALS \cite{TomasiPARAFAC,ComonLS}, but taking more previous iterates into account (we typically use up to 20). However, our approach is also more general in that it can potentially be used to accelerate other CP algorithms than ALS. In fact, the N-GMRES optimization algorithm proposed is based on general concepts and may be applied to other nonlinear optimization problems, as a potential alternative to other first-order optimization methods like N-CG.
\begin{figure}[!htbp]
  \centering
  \scalebox{0.35}{
  \includegraphics{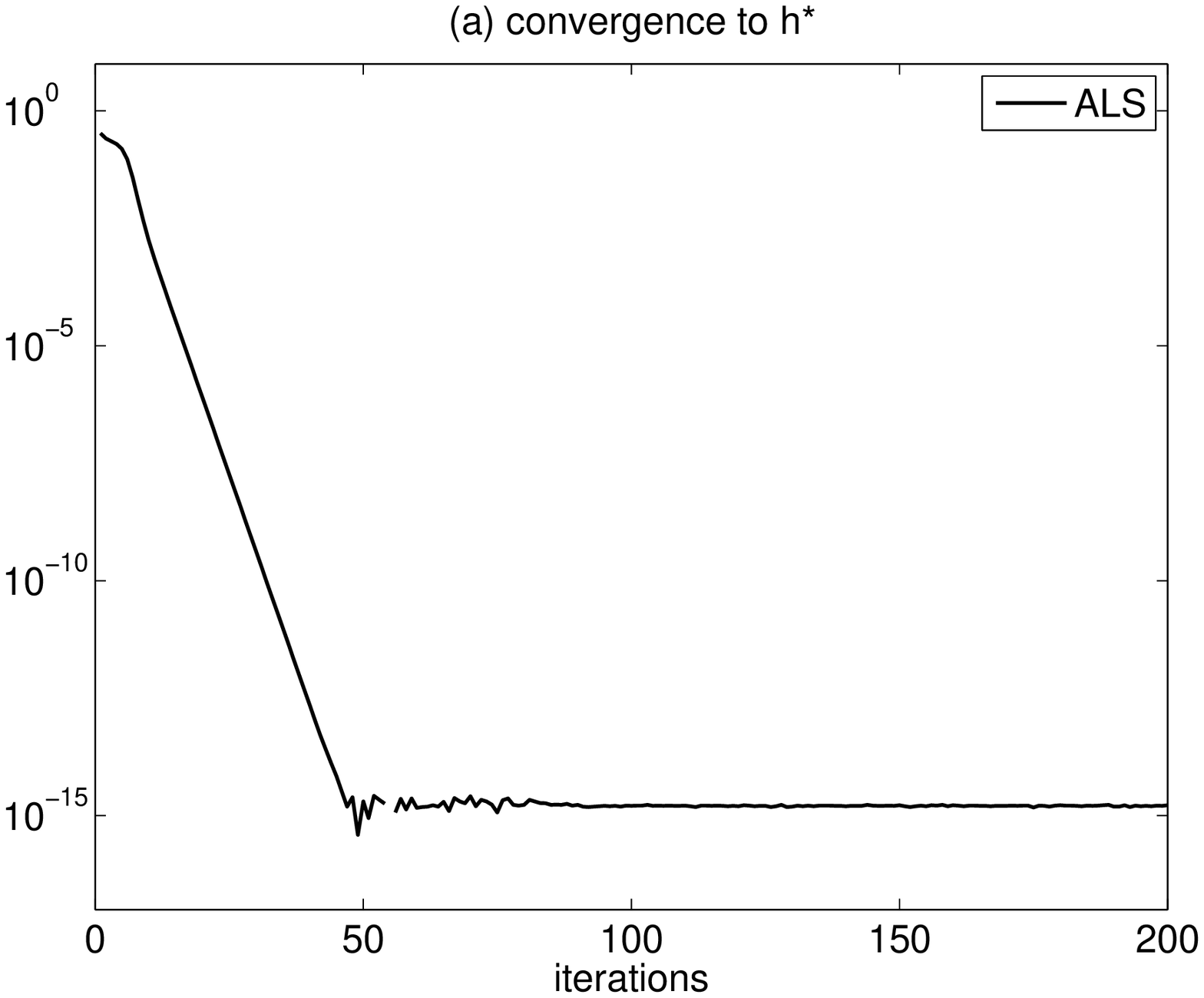}
  \includegraphics{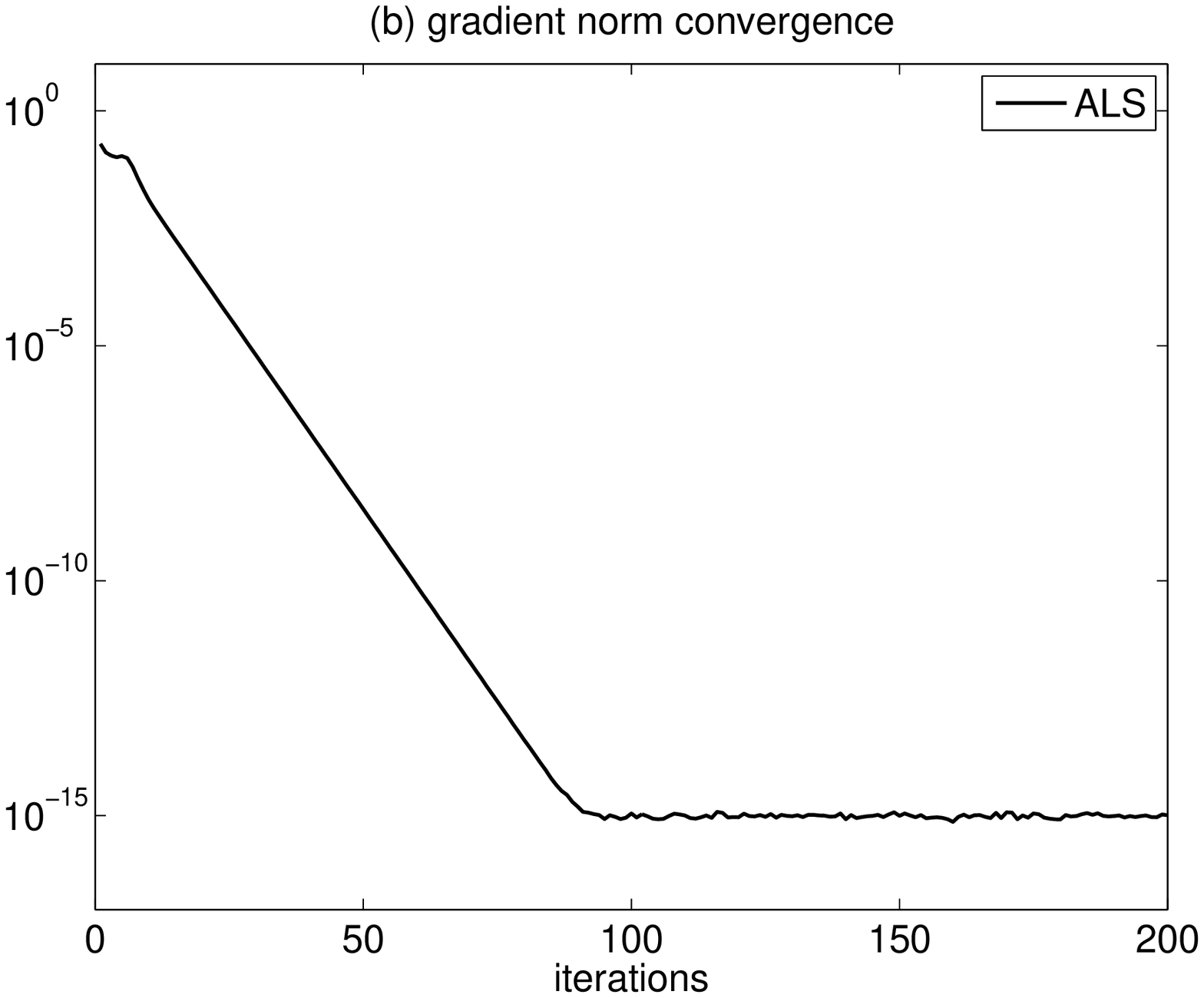}
  }
   \caption{ALS convergence plots for an `easy' instance of Test Problem I (parameters $s=$50, $c=$0.5, $R=$3, $l_{1}=$1, $l_{2}=$1, see Section \ref{sec:NumDense} for full problem description). (a) Convergence of $|h(\mA_R^{(i)})-h^*|$, where $h^*$ is the value of the normalized distance measure, (\ref{eq:hA}), in the stationary point the method converges to. (b) Convergence of the normalized gradient of the objective function, $\|\bg(\mA_R^{(i)})\|_2$/$\|\mT\|_F$, indicating convergence to a stationary point. ALS is a fast method for this `easy' problem, and N-GMRES acceleration is not required.}
   \label{fig:introEasy}
\end{figure}    

\begin{figure}[!htbp]
  \centering
  \scalebox{0.35}{
  \includegraphics{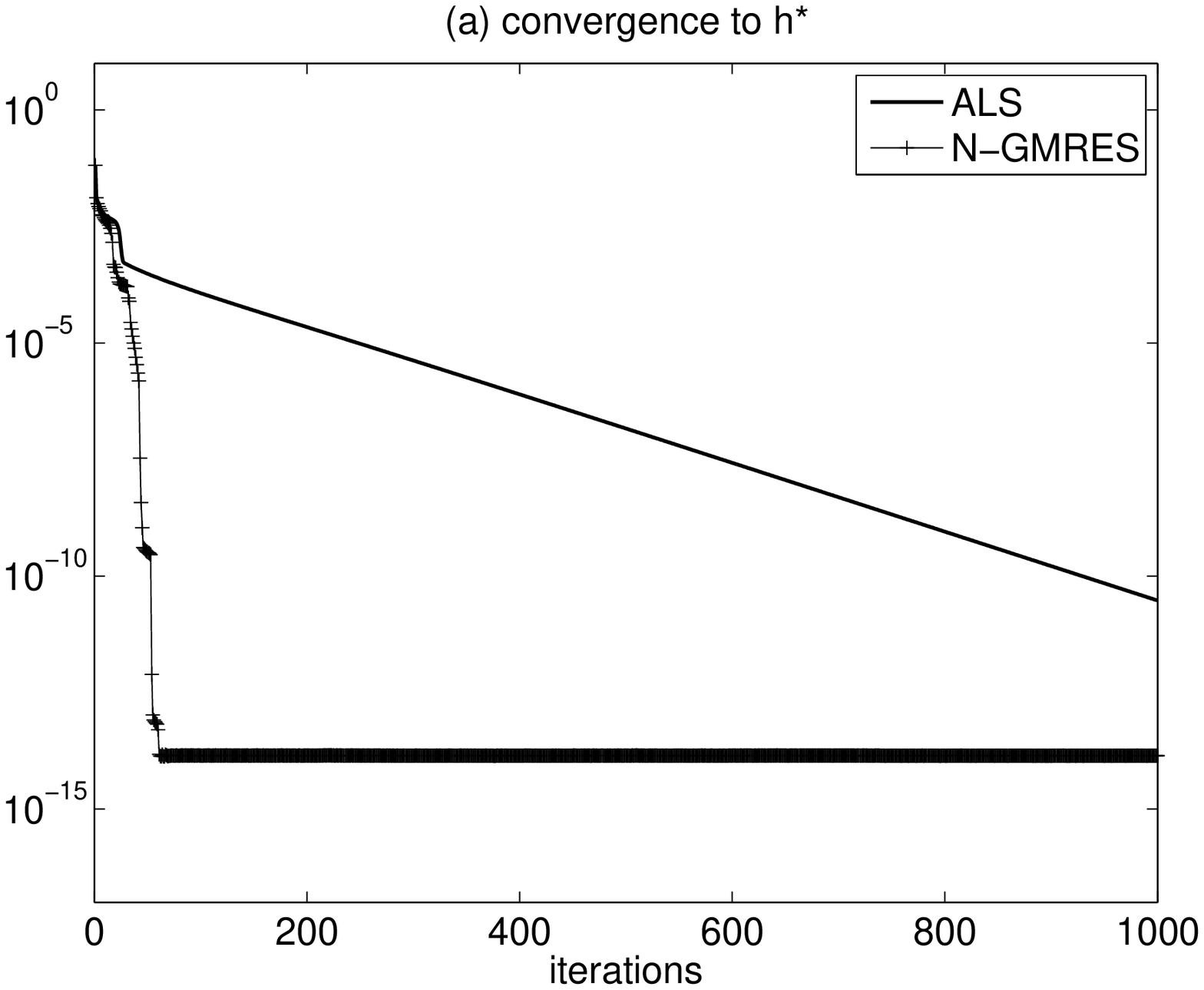}
  \includegraphics{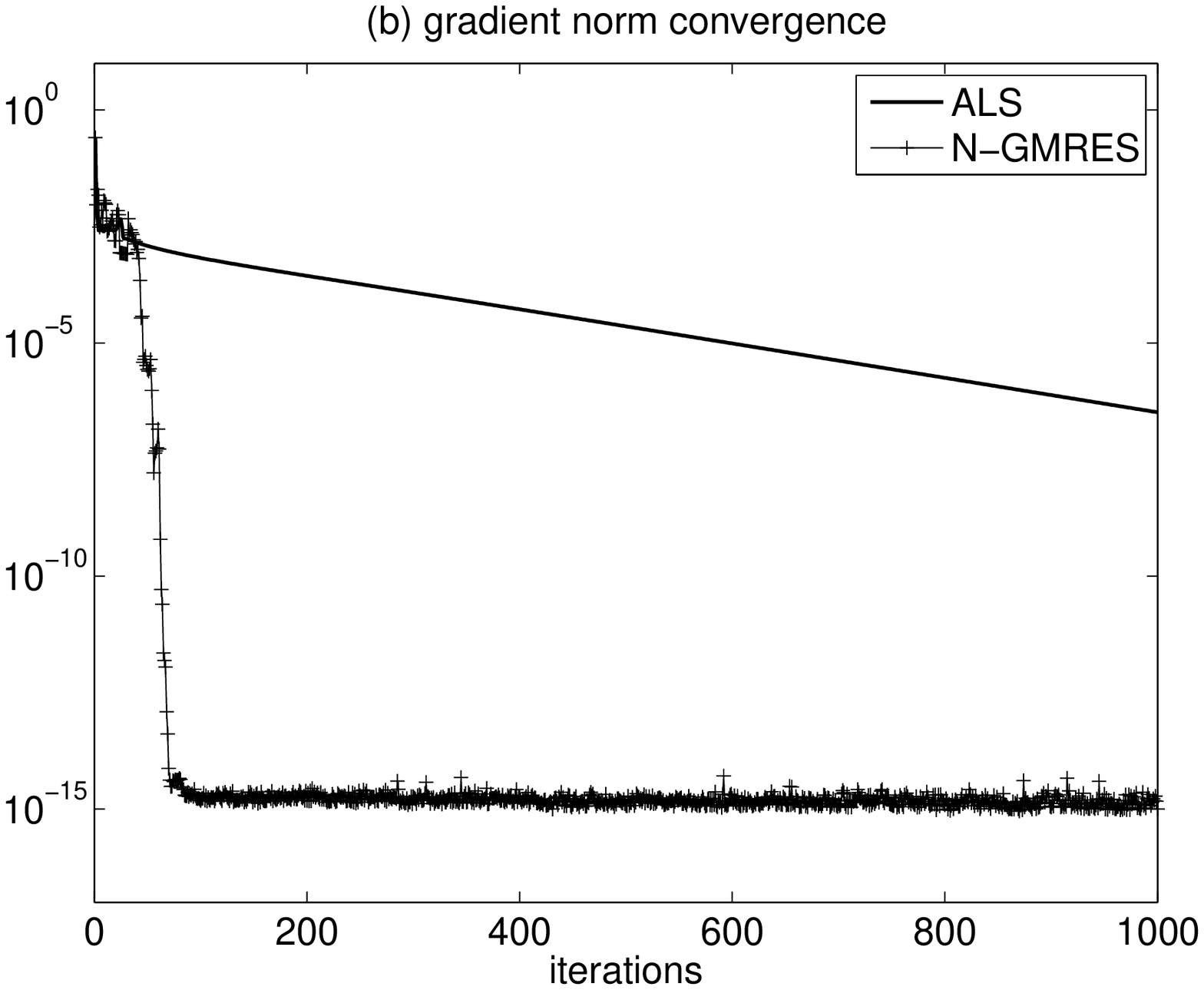}
  }
   \caption{ALS convergence plots for a `difficult' instance of Test Problem I (parameters $s=$50, $c=$0.9, $R=$3, $l_{1}=$1, $l_{2}=$1). (a) Convergence of $|h(\mA_R^{(i)})-h^*|$. (b) Convergence of $\|\bg(\mA_R^{(i)})\|_2$/$\|\mT\|_F$. ALS is slow for this `difficult' problem, but our proposed N-GMRES optimization algorithm significantly reduces the number of iterations.}
   \label{fig:introHard}
\end{figure}    

The nonlinear GMRES acceleration method we propose to use in our algorithm for nonlinear optimization is the nonlinear extension of GMRES that was developed by Washio and Oosterlee in \cite{WashioNGMRES-ETNA} to accelerate multigrid solvers for nonlinear PDE systems (see also \cite{OosterleeNGMRES-SISC,OosterleeNGMRES-ETNA} for further applications of their method in the nonlinear PDE systems context). Their method is a generalization of Saad and Schultz's celebrated GMRES method for linear equation systems \cite{SaadGMRES} to the nonlinear case. In the N-GMRES optimization algorithm we propose, we apply this nonlinear GMRES approach to combine ALS-generated iterates with the goal of driving the residual of the (nonlinear) first-order optimality equations to zero. A line search method is used to provide globalization. Just like for GMRES for linear equation systems, the method that generates the iterates can be seen as a preconditioner for GMRES, or, from an alternative viewpoint, GMRES can be interpreted as an acceleration mechanism for the method that generates the iterates \cite{WashioNGMRES-ETNA,OosterleeNGMRES-SISC}. It will be explained in detail below how this interpretation carries over to the present context of nonlinear optimization for CP decomposition.

Figs.~\ref{fig:introEasy}-\ref{fig:introHardTiming} give a preview of what the N-GMRES optimization algorithm proposed in this paper tries to achieve. The figures present convergence plots for a dense test problem that is a standard test case for CP decomposition from the literature \cite{TomasiPARAFAC,AcarCPOPT}. A 3-way data tensor with size $50\times50\times50$ is generated starting from a rank-3 canonical tensor with random factor matrices modified to have pre-specified collinearity $c$, and noise is added. (See Test Problem I in Section \ref{sec:NumDense} for a detailed description of the test case.). A rank-3 CP approximation is sought starting from a random initial guess. 
In order to track accuracy during the progress of the iterative methods, we define a measure of the normalized distance between data tensor $\mT$ and rank-$R$ approximation $\mA_R^{(i)}$ in iteration $i$ as
\begin{align}
h(\mA_R^{(i)})=\frac{\| \mT  - \mA_R^{(i)}\|_F}{\|\mT\|_F},
\label{eq:hA}
\end{align}
and define the optimal distance
\begin{align}
h^*= h(\mA_R^*),
\label{eq:h*}
\end{align}
where $\mA_R^*$ is the stationary point that the method converges to in the test.
The accuracy of the approximation as the iterative method progresses is then tracked by measuring $|h(\mA_R^{(i)})-h^*|$.
We also track convergence of $\|\bg(\mA_R^{(i)})\|_2$/$\|\mT\|_F$, the norm of the gradient of the objective function normalized by the norm of the data tensor, which gives information about convergence to a stationary point.

Fig.~\ref{fig:introEasy} shows convergence plots for a case with collinearity $c=0.5$. It is known that this case is `easy' for ALS, and the plots confirm that ALS converges quickly to a stationary point. It is also known that factor matrices with nearly collinear columns constitute problems that are more difficult for ALS \cite{TomasiPARAFAC,AcarCPOPT}. Fig.~\ref{fig:introHard} shows convergence plots for a case with $c=0.9$, and we can indeed observe that ALS converges slowly. The plots also show how the N-GMRES optimization algorithm that is proposed in this paper significantly speeds up ALS and reduces the number of iterations dramatically. Of course, Fig.~\ref{fig:introHard} is only part of the story, because the N-GMRES iterations are more expensive than simple ALS iterations. However, the timing plots in Fig.~\ref{fig:introHardTiming} show that significant gains can still be made if this extra cost is taken into account, especially when accurate results are desired.

\begin{figure}[!htbp]
  \centering
  \scalebox{0.35}{
  \includegraphics{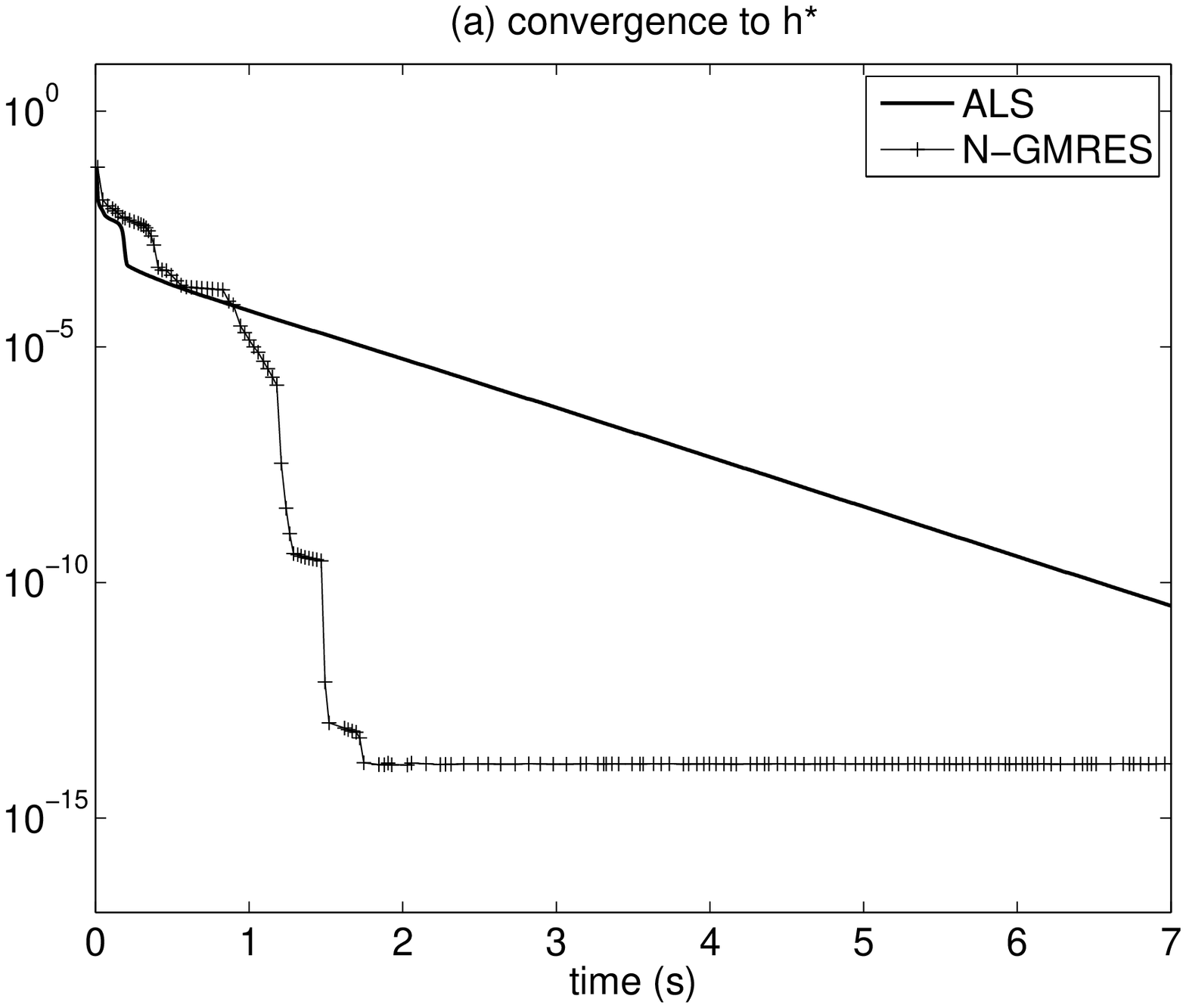}
  \includegraphics{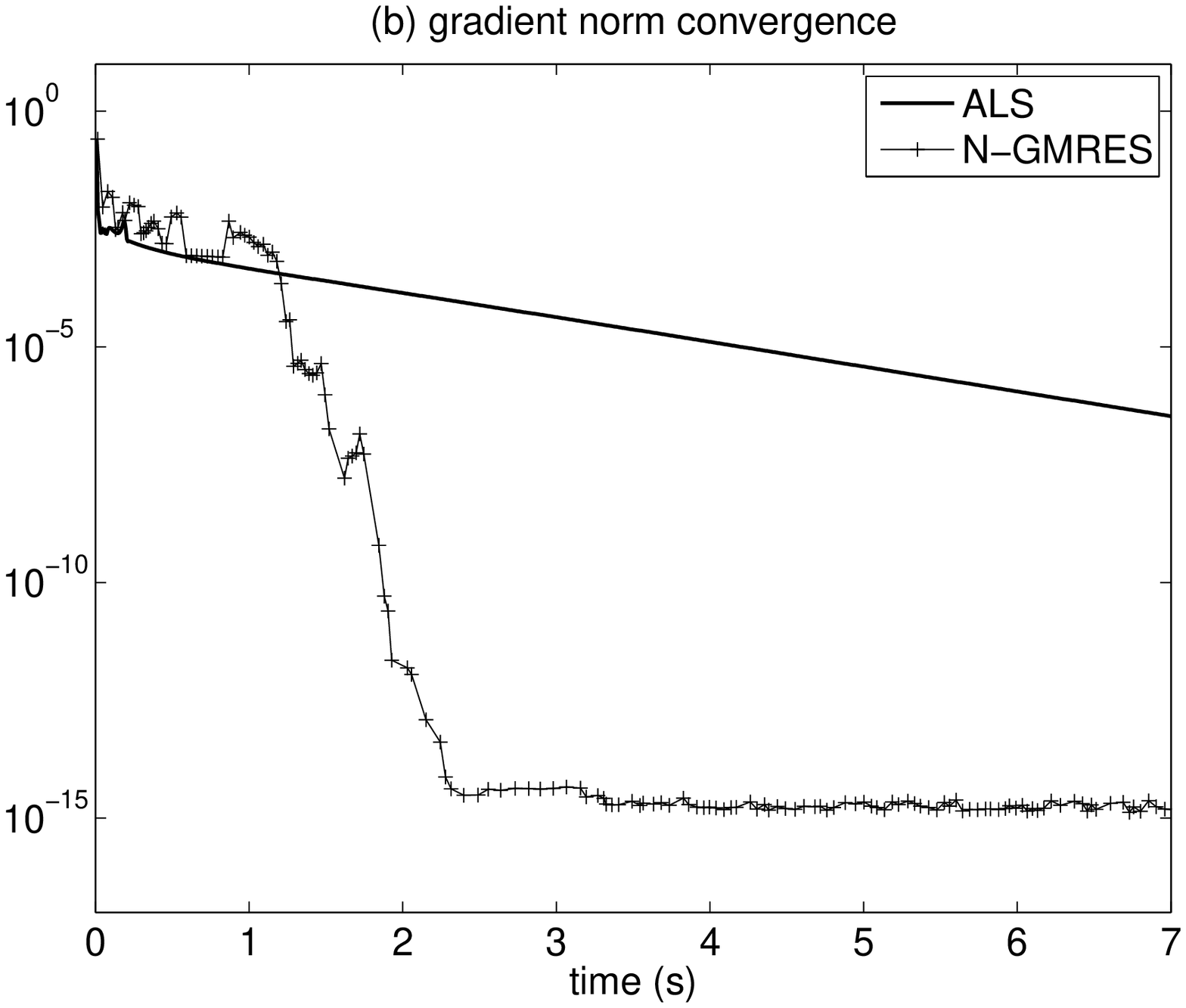}
  }
   \caption{Same as Fig.~\ref{fig:introHard}, but now convergence is plotted as a function of execution time. Even though N-GMRES iterations take more time per iteration, N-GMRES still substantially accelerates ALS, especially if accurate approximations are desired.}
   \label{fig:introHardTiming}
\end{figure}    

The rest of this paper is organized as follows.
In Section \ref{sec:Algorithm} we describe the proposed N-GMRES optimization algorithm for nonlinear optimization problems, and explain how it is applied to compute rank-$R$ canonical tensor decompositions. We also explain the interpretation of the N-GMRES acceleration method as a preconditioned GMRES nonlinear optimization method. In Section \ref{sec:NumDense} we test the N-GMRES optimization algorithm on dense tensor test cases, and in Section \ref{sec:NumSparse} we discuss sparse test problems. Conclusions are formulated in Section \ref{sec:conc}.

\section{N-GMRES Optimization Algorithm for Canonical Tensor Decomposition}
\label{sec:Algorithm}

In this section, we describe the proposed N-GMRES optimization algorithm for nonlinear optimization problems, and explain how it is applied to compute rank-$R$ canonical tensor decompositions. We start with a description of the N-GMRES optimization approach, and situate this discussion in the context of a general nonlinear optimization problem, since the approach is general.
We then compare the N-GMRES approach for optimization to GMRES for linear systems, explaining how the dual viewpoint of preconditioned GMRES and GMRES acceleration applies in the nonlinear optimization context. In our descriptions we closely follow the derivations and presentation of Washio and Oosterlee's nonlinear GMRES for PDE problems from \cite{WashioNGMRES-ETNA,OosterleeNGMRES-SISC}, which we propose to use as the main ingredient in our nonlinear optimization algorithm. We give extensive details because these details give precise insight into how the GMRES method generalizes to the nonlinear optimization problem. This is followed by a discussion of how N-GMRES is applied to CP optimization, giving details on the first-order optimality equations for CP and the line search algorithm we use in the numerical results of subsequent sections.

\subsection{N-GMRES Optimization Algorithm}
\label{subsec:N-GMRES}
Consider the following nonlinear optimization problem with associated first-order optimality equations:\\

\noindent
{\sc optimization problem II:}
\begin{align}
\text{find $\bu^*$ that minimizes }f(\bu).
\label{eq:fu}
\end{align}

\noindent
{\sc first-order optimality equations II:}
\begin{align}
\nabla f(\bu)=\bg(\bu)=0.
\label{eq:gu}
\end{align}
Assume that we have $i+1$ previous iterates $\bu_0,\bu_1,\ldots,\bu_{i-1},\bu_i$, and that we have a one-step (nonlinear) iterative update process $M(.)$ that generates a new approximation from an existing approximation:
\begin{align}
\bu_{new}=M(\bu_{old}).
\label{eq:Mu}
\end{align}
The goal of the N-GMRES optimization algorithm will be to accelerate the convergence of iterative update process $M(.)$.

Every iteration of the N-GMRES optimization algorithm consists of three steps.\\

\noindent
{\sc Step I:}

In the first step, we generate a preliminary new iterate $\bar{\bu}_{i+1}$ from the last iterate $\bu_i$ using one-step iterative update process $M(.)$:
\begin{align}
\bar{\bu}_{i+1}=M(\bu_{i}).
\label{eq:Mub}
\end{align}

\noindent
{\sc Step II:}

In the second step, we find an accelerated iterate $\hat{\bu}_{i+1}$ that is obtained by recombining previous iterates as follows:
\begin{align}
\hat{\bu}_{i+1}=\bar{\bu}_{i+1}+\sum_{j=0}^{i} \, \alpha_j \, (\bar{\bu}_{i+1}-\bu_j).
\label{eq:accel}
\end{align}
The N-GMRES algorithm seeks to determine the unknown coefficients $\alpha_j$ such that the two-norm of the gradient of the objective function evaluated at the accelerated iterate is small. We call $\bg(\hat{\bu}_{i+1})$ the residual at $\hat{\bu}_{i+1}$, and the objective is thus to determine the $\alpha_j$ such that the residual is minimized in the two-norm: we attempt to minimize $\|\bg(\hat{\bu}_{i+1})\|_2$. However, $\bg(.)$ is a nonlinear function of the $\alpha_j$ (more precisely, it is a high-order polynomial in the $\alpha_j$), and we proceed by linearization to allow for inexpensive computation of coefficients $\alpha_j$ that may approximately minimize $\|\bg(\hat{\bu}_{i+1})\|_2$. 
Using the following approximations
\begin{align}
\bg(\hat{\bu}_{i+1})&\approx \bg(\bar{\bu}_{i+1})+\sum_{j=0}^{i} \, \left. \frac{\partial \bg}{\partial \bu} \right|_{\bar{\bu}_{i+1}} \, \alpha_j \, (\bar{\bu}_{i+1}-\bu_{j}) \nonumber\\
  & \approx \bg(\bar{\bu}_{i+1})+\sum_{j=0}^{i} \, \alpha_j \, (\bg(\bar{\bu}_{i+1})-\bg(\bu_{j}))
  \label{eq:linearize}
\end{align}
we arrive at minimization problem
\begin{gather*}
\text{find coefficients $(\alpha_0, \ldots, \alpha_i)$ that minimize }\\
 \| \bg(\bar{\bu}_{i+1})+\sum_{j=0}^{i} \, \alpha_j \, (\bg(\bar{\bu}_{i+1})-\bg(\bu_{j})) \|_2.
\label{eq:minAlpha}
\end{gather*}
This is a standard least-squares problem that can be solved, for example, by using the normal equations.
Defining
\begin{align*}
 \boldsymbol{\alpha}&=(\alpha_0,\ldots,\alpha_i)^T,\\
 \bp_j&=\bg(\bar{\bu}_{i+1})-\bg(\bu_{j}),\\
 \bP&=\left[ \bp_0 | \ldots | \bp_j \right],
\end{align*}
we minimize $\|\bP\, \boldsymbol{\alpha} + \bg(\bar{\bu}_{i+1}) \|_2$, which leads to normal equation system
\begin{align}
\bP^T \, \bP \, \boldsymbol{\alpha}=-\bP^T \, \bg(\bar{\bu}_{i+1}).
\label{eq:normal}
\end{align}

\noindent
{\sc Step III:}

In the third step, we perform a line search that optimizes objective function $\bff(\bu)$ on a half line starting at preliminary iterate $\bar{\bu}_{i+1}$, which was generated in Step I, and connecting it with accelerated iterate $\hat{\bu}_{i+1}$, which was generated in Step II:
\begin{gather}
\text{find $\beta^* \in [0,\infty)$ that minimizes }\\
\bff(\bar{\bu}_{i+1}+\beta(\hat{\bu}_{i+1}-\bar{\bu}_{i+1})).
\label{eq:linesearch}
\end{gather}
This line search step is necessary, because without it erratic convergence behavior can occur, especially as long as iterates are far from a stationary point. In this case, the linearization steps in (\ref{eq:linearize}) may incur large errors, resulting in accelerated iterates $\hat{\bu}_{i+1}$ that may be bad approximations.
The result of the line search is finally selected as the new iterate $\bu_{i+1}$:
\begin{align}
\bu_{i+1}=\bar{\bu}_{i+1}+\beta^* (\hat{\bu}_{i+1}-\bar{\bu}_{i+1}).
\label{eq:new}
\end{align}

In practice, to limit the computational cost and especially the memory cost of the N-GMRES acceleration, we implement the algorithm with a windowed acceleration process with window size $w$, in which N-GMRES-accelerated iterates are generated based on the $w$ last iterations. 

Fig.~\ref{fig:N-GMRES} gives a schematic representation of the N-GMRES optimization algorithm, and Algorithm \ref{alg:N-GMRES} describes it in pseudo-code. (Note of course that the $w$ initial iterates required in the pseudo-code description can naturally be generated by applying the algorithm with a window size that gradually increases from one up to $w$, starting from a single initial guess.)
\begin{figure}[!htbp]
  \centering
  \scalebox{1.2}{
    \includegraphics{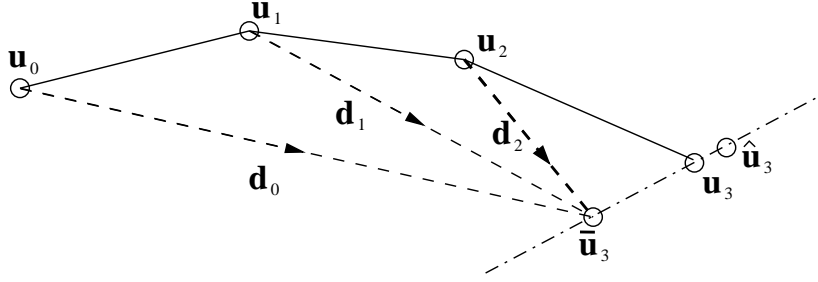}
  }
   \caption{Schematic representation of one iteration of the N-GMRES optimization algorithm. Given previous iterations $\bu_0$, $\bu_1$ and $\bu_2$, new iterate $\bu_3$ is generated as follows. In Step I, preliminary iterate $\bar{\bu}_3$ is generated by the one-step update process $M(.)$: $\bar{\bu}_3=M(\bu_2)$. In Step II, the nonlinear GMRES step, accelerated iterate $\hat{\bu}_3$ is obtained by determining the coefficients $\alpha_j$ in $\hat{\bu}_3=\bar{\bu}_3+\alpha_0 \bd_0+\alpha_1 \bd_1+\alpha_2 \bd_2$ such that the gradient of the objective function in $\hat{\bu}_3$ is approximately minimized. In Step III, the new iterate, $\bu_3$, is finally generated by a line search that minimizes the objective function $\bff(\bar{\bu}_{3}+\beta(\hat{\bu}_{3}-\bar{\bu}_{3}))$.}
   \label{fig:N-GMRES}
\end{figure}    

\begin{algorithm}[H]
\dontprintsemicolon

{\bf Input:} $w$ initial iterates $\bu_0, \ldots,\bu_{w-1}$.\;
\ \\
$i=w-1$\;
\Repeat{\text{convergence criterion satisfied}}{
{\sc Step I:} {\em (generate preliminary iterate by one-step update process $M(.)$)}\;
$\qquad \bar{\bu}_{i+1}=M(\bu_{i})$\;
{\sc Step II:}  {\em (generate accelerated iterate by nonlinear GMRES step)}\;
$\qquad \hat{\bu}_{i+1}=$gmres$(\bu_{i-w+1},\ldots,\bu_{i};\bar{\bu}_{i+1})$\;
{\sc Step III:}  {\em (generate new iterate by line search process)}\;
$\qquad \bu_{i+1}=$linesearch$(\bar{\bu}_{i+1}+\beta(\hat{\bu}_{i+1}-\bar{\bu}_{i+1}))$\;
$i=i+1$\;
}\;
\ \\
\caption{N-GMRES optimization algorithm (window size $w$)}
\label{alg:N-GMRES}
\end{algorithm}

We now give some additional remarks about the N-GMRES optimization algorithm proposed and its application to canonical tensor decomposition.

First, for Step I in the algorithm, we use the ALS algorithm as the one-step update process $M(.)$. The ALS algorithm for CP decomposition will be described in Section \ref{subsec:optimality}, along with specific expressions for the gradient of the objective function and the first-order optimality equations of Optimization Problem I. The one-step update process $M(.)$ can be interpreted as the preconditioner of the nonlinear GMRES procedure, as will be explained in Section \ref{subsec:GMRES}.

Second, Step I and Step II in the N-GMRES optimization algorithm are entirely analogous to the nonlinear GMRES approach for PDE systems that was proposed by Washio and Oosterlee in \cite{WashioNGMRES-ETNA}. In their applications, they employ nonlinear GMRES to drive the residuals of the PDEs to zero, and use multigrid as the preconditioner \cite{WashioNGMRES-ETNA,OosterleeNGMRES-SISC,OosterleeNGMRES-ETNA}. In the present context of canonical tensor decomposition as a nonlinear optimization problem, we drive the residual of the first-order optimality equations (the gradient of the objective function) to zero, and use ALS as the preconditioner.

Third, Step III in the N-GMRES optimization algorithm performs a role analogous to the selection mechanism proposed by Washio and Oosterlee in \cite{WashioNGMRES-ETNA} to reduce erratic convergence. They select either the preliminary iterate or the accelerated one, based on ingenious but somewhat ad-hoc criteria that consider changes in the norms of the residuals and in the solution updates. Instead, in the present context of nonlinear optimization problems, we can exploit the final goal of minimizing the objective function to control erratic behavior, and rather than selecting either the preliminary iterate or the accelerated one, we perform a line search on the line connecting the two. Not only does this control erratic behavior, but it also provides an improved new iterate, at the cost of a line search.

Fourth, it may be beneficial to restart the N-GMRES acceleration with window size one when indications arise that the current window contains iterates that harm the acceleration process. In \cite{WashioNGMRES-ETNA} the GMRES procedure is restarted based on criteria similar to the ones that are used to choose between the preliminary or accelerated iterate. In the context of N-GMRES for nonlinear optimization problems, we propose the following simple criterion: we restart whenever the search direction of the line search, $\hat{\bu}_{i+1}-\bar{\bu}_{i+1}$, does not point in a descent direction (so we set the window size back to one whenever $\bg(\bar{\bu}_{i+1})^T \, (\hat{\bu}_{i+1}-\bar{\bu}_{i+1})>0$). The motivation for this is simple: we expect the acceleration mechanism to be effective close to a stationary point. In that case, the accelerated iterate is expected to be an improvement over the preliminary iterate, and the search direction is expected to be a descent direction. If this is not the case, this is taken as an indication that the current N-GMRES sequence does not help, and N-GMRES is restarted. Our numerical tests have indicated that this simple approach indeed tends to be beneficial for speed of convergence, for the test problems we considered. Note that windowing with restarts can be implemented efficiently in an elegant way, and we refer to the detailed pseudocode in \cite{WashioNGMRES-ETNA} for this.

Fifth, another practical point is that the normal equation system in (\ref{eq:normal}) may become ill-conditioned since the vectors $\bp_j$ may become nearly linearly dependent. One way to deal with this, as suggested in \cite{WashioNGMRES-ETNA}, is to add a small multiple of the identity matrix, $\delta \, \bI$, to the normal equation operator. This was sufficient for the numerical tests we performed, and \cite{WashioNGMRES-ETNA} provides further theoretical justification for this fix.

Sixth, for the line search of Step III, we use in our numerical experiments the Mor\'{e}-Thuente line search method
\cite{MoreThuente} as implemented in the Poblano toolbox for Matlab \cite{POBLANO}. This line search method was also used in the N-CG method for canonical tensor decomposition in \cite{AcarCPOPT}. This is a rather sophisticated line search method that employs multiple function and gradient evaluations to improve the solution.

Finally, we briefly discuss the computational cost of the N-GMRES optimization algorithm.
In terms of memory cost, for window size $w$ the algorithm requires storing $w$ previous solution vectors and residual vectors. In the numerical results presented below we will see that, for the CP decomposition test problems we consider, a window size of approximately $w=20$ is optimal in terms of computing time. This requires substantial additional memory, however, and memory can obviously be traded for speed.
In terms of the computing time, the dominant costs in each iteration of N-GMRES nonlinear optimization algorithm \ref{alg:N-GMRES} are applying $M(.)$ in Step I, computing the gradient of the preliminary iterate, $\bg(\bar{\bu}_{i+1})$, in Step II, and the function and gradient evaluations in the line search of Step III. Since the optimization problems we solve in this paper are quite involved and the ALS, function and gradient evaluations are very expensive, these calculations strongly dominate all others in practice. The relevant extra costs of the N-GMRES optimization approach compared to just applying the one-step update process $M(.)$ are thus the additional function and gradient evaluations, and the cost of building and solving the normal equation system (\ref{eq:normal}) and computing the accelerated iterate is negligible in practice. The number of function and gradient evaluations per line search depend on the line search method used and the accuracy parameters one chooses for the line search calculations, which in turn influence the number of overall iterations required to achieve a certain accuracy for the stationary point.
It is thus difficult to say much in general about the additional cost of the line search step. The numerical results to be presented below show that, for two relevant classes of test problems for canonical tensor decomposition, N-GMRES optimization can be significantly faster than stand-alone iteration by applying $M(.)$, especially when high accuracy is desired. We will at this point just give one example. For the N-GMRES CP calculation shown in Fig.~\ref{fig:introHardTiming}, up to an accuracy of $|h-h^*|<10^{-10}$ (which required 54 iterations), about 40\% of the time is spent in the ALS iterations (54 calls), 15\% in the calculation of $\bg(\bar{\bu}_{i+1})$ (54 calls), and about 30\% in the function and gradient evaluations in the line search procedure (137 calls to evaluate $\bff$ and $\bg$). The average number of function/gradient evaluations per line search was about 2.5. While the generality of this single example is of course limited, it does illustrate that N-GMRES optimization algorithm \ref{alg:N-GMRES} mainly adds overhead in terms of additional function and residual evaluations for the GMRES and the line search steps. But it is precisely these additional calculations that potentially lead to a great reduction in the number of iterations and overall execution time, as we will illustrate with extensive numerical tests below.

\subsection{Interpretation as a GMRES Method: Acceleration and Preconditioning}
\label{subsec:GMRES}
In this subsection we briefly recall some particulars of the GMRES method for linear equation system
\begin{align}
\bA \, \bu=\bb,
\label{eq:Aub}
\end{align}
and give a detailed exposition of the parallels with the N-GMRES optimization algorithm proposed in the previous subsection.
We follow the presentation of \cite{WashioNGMRES-ETNA}, but explain in detail how the parallels apply in the nonlinear optimization context. As in \cite{WashioNGMRES-ETNA}, we discuss GMRES for (\ref{eq:Aub}) in a particular, perhaps non-standard way, which will allow us to draw parallels to the nonlinear optimization algorithm.

Our starting point for explaining the principles of preconditioned GMRES for linear equation system (\ref{eq:Aub}) is to consider so-called stationary iterative methods for (\ref{eq:Aub}) of the following form:
\begin{align}
\bu_{i+1}=\bu_i+\bM^{-1}\, \br_i.
\label{eq:stat}
\end{align}
Here, $\br_i$ is the $i$th residual, defined by
\begin{align}
\br_i=\bb-\bA\,\bu_i,
\label{eq:resid}
\end{align}
and matrix $\bM$ is an approximation of $\bA$ that has an easily computable inverse, i.e., $\bM^{-1}\approx\bA^{-1}$.
For example, $\bM$ can be chosen to correspond to Gauss-Seidel or Jacobi iteration, or to a multigrid cycle \cite{WashioNGMRES-ETNA}.
General form (\ref{eq:stat}) of a stationary iterative method can be motivated as follows. Defining the error of the $i$th iterate as $\be_i=\bu-\bu_i$ (with $\bu$ the exact solution of (\ref{eq:Aub})), it is easy to see that $\bA\,\be_i=\br_i$. Observing that $\bu=\bu_i+\be_i=\bu_i+\bA^{-1}\,\br_i$, and that knowledge of $\bA^{-1}$ would (obviously) lead to the exact solution in one step, one can conclude that, if one does not know $\bA^{-1}$ but does have access to an easily computable approximation $\bM^{-1}$, update formula (\ref{eq:stat}) may lead to a useful iteration process.

Consider a sequence of iterates $\bu_0, \ldots,\bu_i$ generated by update formula (\ref{eq:stat}), starting from some initial guess $\bu_0$. Note that the residuals of these iterates are related as follows:
\begin{align}
 \br_i&=\bb-\bA\,\bu_i \nonumber \\
        &=(\bI-\bA\bM^{-1})\,\br_{i-1} \nonumber \\
        &=(\bI-\bA\bM^{-1})^i\,\br_0. 
\label{eq:residrel}
\end{align}
This motivates the definition of the following vector spaces:
\begin{align}
V_{1,i+1}&=\mathop{span}\{ \br_0,\ldots,\br_i\}, \nonumber\\
V_{2,i+1}&=\mathop{span}\{ \br_0, \bA \bM^{-1} \, \br_0, (\bA \bM^{-1})^2 \, \br_0\}, \ldots, (\bA \bM^{-1})^i \, \br_0\} \nonumber\\
  &=K_{i+1}(\bA \bM^{-1},\br_0), \label{eq:V2}\\
V_{3,i+1}&=\mathop{span}\{ \bM \, (\bu_1-\bu_0),  \bM \, (\bu_2-\bu_1),\ldots,  \bM \, (\bu_{i+1}-\bu_i) \},\nonumber \\
V_{4,i+1}&=\mathop{span}\{ \bM \, (\bu_{i+1}-\bu_0),  \bM \, (\bu_{i+1}-\bu_1),\ldots,  \bM \, (\bu_{i+1}-\bu_i) \}.\label{eq:V4}
\end{align}
Vector space $V_{2,i+1}$ is the so-called Krylov space $K_{i+1}(\bA \bM^{-1},\br_0)$ of order $i+1$, generated by matrix $\bA \bM^{-1}$ and vector $\br_0$.
It is easy to see that the vector spaces defined above are equal \cite{WashioNGMRES-ETNA}:
\begin{lemma}{$V_{1,i+1}=V_{2,i+1}=V_{3,i+1}=V_{4,i+1}.$}
\label{lem:spaces}
\end{lemma}
\begin{proof}
First, $V_{1,i+1}=V_{2,i+1}$ by (\ref{eq:residrel}), which directly shows that $\br_j \in V_{2,i+1}$ for all $j$, and $(\bA \bM^{-1})^j \, \br_0 \in V_{1,i+1}$ for all $j$ follows by recursion.\\
Second, $V_{2,i+1}=V_{3,i+1}$ because $M\,(\bu_{i+1}-\bu_i)=\br_i$, by (\ref{eq:stat}).\\
Third, $V_{3,i+1}=V_{4,i+1}$ because, for $k<i+1$, $\bu_{i+1}-\bu_k=\sum_{j=k+1}^{i+1} \, (\bu_j-\bu_{j-1})$, and
$\bu_k-\bu_{k-1}=(\bu_{i+1}-\bu_{k-1})-(\bu_{i+1}-\bu_k).$
\end{proof}

We know that $\bM \, (\bu_{i+1}-\bu_i) \in K_{i+1}(\bA \bM^{-1},\br_0)$. The GMRES procedure can be seen as a way to accelerate stationary iterative method (\ref{eq:stat}), by recombining iterates (or, equivalently, by reusing residuals). In particular, we seek a better approximation $\hat{\bu}_{i+1}$, with $\bM \, (\hat{\bu}_{i+1}-\bu_i)$ in the Krylov space $K_{i+1}(\bA \bM^{-1},\br_0)$, such that $\hat{\br}_{i+1}=\bb-\bA\,\hat{\bu}_{i+1}$ has minimal two-norm.
In other words, we seek optimal coefficients $\beta_j$ in
\begin{align*}
\bM \, (\hat{\bu}_{i+1}-\bu_i) &= \sum_{j=0}^{i} \, \beta_j \, \bM \, (\bu_{i+1}-\bu_j),\\
\end{align*}
or in
\begin{align*}
\hat{\bu}_{i+1}&=\bu_i + \sum_{j=0}^{i} \, \beta_j \, (\bu_{i+1}-\bu_j)\\
 &= \bu_{i+1} - (\bu_{i+1}-\bu_i)+ \sum_{j=0}^{i} \, \beta_j \, (\bu_{i+1}-\bu_j).
\end{align*}
Equivalently, we can seek optimal coefficients $\alpha_j$ in
\begin{align}
\hat{\bu}_{i+1}&=\bu_{i+1} + \sum_{j=0}^{i} \, \alpha_j \, (\bu_{i+1}-\bu_j),
\label{eq:GMRESopt}
\end{align}
such that $\|\hat{\br}_{i+1}\|_2$ is minimized (which leads to a small least-squares problem). Apart from the power of the minimization principle, GMRES for linear systems is especially powerful because this minimization can be done very efficiently, using the Arnoldi iteration to generate orthogonal bases for Krylov spaces of increasing order \cite{SaadGMRES}.
For nonlinear problems, such an advantageous implementation is not possible, but the general ideas of optimal acceleration remain powerful and can still be applied.

We now explain two complementary viewpoints of GMRES for linear system (\ref{eq:Aub}). We have presented the method as a way to accelerate simple one-step stationary iterative method (\ref{eq:stat}). A more customary way to see GMRES is in terms of preconditioning. With $\bM=\bI$, the approach described above reduces to `non-preconditioned' GMRES. Preconditioned GMRES can then also be derived by applying non-preconditioned GMRES to the preconditioned linear equation system \linebreak
$\bA \bM^{-1} (\bM \bu)=\bb$. In this viewpoint, the matrix $\bM^{-1}$ is called the preconditioner matrix, because its role is viewed as to pre-condition the spectrum of the linear system operator such that the (non-preconditioned) GMRES method applied to 
$(\bA \bM^{-1}) \by=\bb$ becomes more effective. By extension, it is also customary to say that the stationary iterative method preconditions GMRES (in the sense of, for example, Gauss-Seidel or multigrid being used as preconditioners for GMRES). In this viewpoint, the role of the stationary iterative method is to generate a preconditioned residual that builds the Krylov space.

The stage is now set for discussing the interpretation of the N-GMRES optimization algorithm of the previous subsection as a preconditioned GMRES method (along the lines of \cite{WashioNGMRES-ETNA} for nonlinear GMRES applied to nonlinear PDE systems).
Before we do so, we need to add one important remark regarding GMRES for linear systems. In the presentation above, all iterates $\bu_j$ for $j=0,\ldots,i$ (for instance, in the right-hand side of (\ref{eq:GMRESopt})) refer to iterates generated by stationary iterative method (\ref{eq:stat}), without acceleration. However, the formulas remain valid when we use accelerated iterates instead; this merely changes the values of the coefficients $\alpha_j$, but leads to the same accelerated iterates \cite{WashioNGMRES-ETNA}. The reason is that all the GMRES optimization does is to produce improved iterates with residuals that are optimal elements of the Krylov spaces, but the Krylov spaces are still the ones generated by the residuals of the stationary iterative method, and do themselves not change, due to linearity.

The N-GMRES optimization algorithm presented in Subsection \ref{subsec:N-GMRES} can be interpreted as a preconditioned GMRES method as follows. 
With $\bu_{i+1}$ the preliminary $i+1$st iterate generated by (\ref{eq:stat}) and the $\bu_j$ with $j\le i$ the previous accelerated iterates, the expression for the accelerated iterate $\hat{\bu}_{i+1}$ we seek in (\ref{eq:GMRESopt}) for the linear case corresponds directly to the expression in (\ref{eq:accel}) for the nonlinear case. The definition of $V_{2,i+1}$ in (\ref{eq:V2}) does not provide a nonlinear generalization for the Krylov space in which we seek an optimal approximation in the linear case, but the definition of $V_{4,i+1}$ (\ref{eq:V4}) does: the image of the linear Krylov space $K_{i+1}(\bA \bM^{-1},\br_0)$ under the preconditioning matrix $\bM^{-1}$ is $\mathop{span}\{(\bu_{i+1}-\bu_0),(\bu_{i+1}-\bu_1),\ldots,(\bu_{i+1}-\bu_i)\}$, and this can trivially be generalized to the nonlinear case: the nonlinear generalization of the Krylov space is precisely given by this vector space. And it is precisely in this space that we seek an optimal vector, both in the linear and in the nonlinear case (see (\ref{eq:GMRESopt}) and (\ref{eq:accel}), respectively). Stationary iterative method (\ref{eq:stat}) is the preconditioning process for GMRES in the linear case, responsible for generating a preliminary approximation with the help of a residual calculation and a preconditioning matrix, and in the same way one-step approximation method (\ref{eq:Mub}) is the preconditioner for GMRES in the nonlinear case. In the present case of the N-GMRES optimization algorithm applied to canonical tensor decomposition, we can say that ALS is used as the preconditioner of N-GMRES. The alternative viewpoint is that N-GMRES accelerates ALS.

A variety of preconditioners exist for linear systems of equations, and research in this area has been a very active field ever since GMRES was proposed. In this paper, we accelerate ALS for the canonical tensor decomposition optimization problem, but the preconditioning interpretation of the N-GMRES optimization algorithm suggests that it may be worthwhile to explore different N-GMRES preconditioners for the canonical tensor decomposition optimization problem in future work. Similarly, it may be interesting to explore preconditioned N-GMRES algorithms for other nonlinear optimization problems. (Or, put in another way, there appears to be potential for N-GMRES to accelerate already existing iterative methods for other nonlinear optimization problems.)


\subsection{First-Order Optimality Equations for the CP Optimization Problem}
\label{subsec:optimality}
In this subsection, we briefly recall the first-order optimality equations for CP Optimization Problem I, following 
\cite{AcarCPOPT}. The gradient of objective function (\ref{eq:fA}) can be written as a vector of matrices
\begin{align}
\nabla f(\mA_R)=\vec{\bG}(\mA_R)=(\bG^{(1)},\ldots,\bG^{(N)}),
\label{eq:GA}
\end{align}
with $\bG^{(n)} \in \mathbb{R}^{I_n \times R}$ ($n=1,\ldots,N$). The matrices $\bG^{(n)}$ for $n=1,\ldots,N$ are given by
\begin{align}
\bG^{(n)}=-\bbT_{(n)} \bar{\boldsymbol{\Phi}}^{(n)}+\bA^{(n)}\, \bar{\boldsymbol{\Gamma}}^{(n)},
\label{eq:Gn}
\end{align}
with variables defined as follows.

With $\mT$ and $\mA_R \in \mathbb{R}^{I_1 \times \ldots \times I_N}$, define size parameters
\begin{align}
K&=\prod_{l=1}^{N} \, I_l,\\
\bar{K}^{(n)}&=\prod_{
\substack{
      l=1 \\
      l \ne n}
}^{N} \, I_l.
\label{eq:K}
\end{align}
Then $\bbT_{(n)} \in \mathbb{R}^{I_n \times \bar{K}^{(n)}}$ is the mode-$n$ matricization of $\mT$, obtained by stacking the $n$-mode fibres of $\mT$ in its columns in a regular way, see \cite{KoldaSIREV}.
Similar to size parameters $K$ and $\bar{K}^{(n)}$, we define matrices $\boldsymbol{\Phi} \in \mathbb{R}^{K \times R}$ and $\bar{\boldsymbol{\Phi}}^{(n)} \in
\mathbb{R}^{\bar{K}^{(n)} \times R}$ by
\begin{align}
\boldsymbol{\Phi}&=\bA^{(1)} \odot \ldots \odot \bA^{(N)},\\
\bar{\boldsymbol{\Phi}}^{(n)}&=\bA^{(1)} \odot \ldots \odot \bA^{(n-1)} \odot \bA^{(n+1)} \odot \ldots \odot \bA^{(N)},
\label{eq:Phi}
\end{align}
where $\odot$ is the Khatri-Rao product \cite{KoldaSIREV}.
Finally, the matrices $\boldsymbol{\Gamma} \in \mathbb{R}^{R \times R}$ and $\bar{\boldsymbol{\Gamma}}^{(n)} \in
\mathbb{R}^{R \times R}$ are defined by
\begin{align}
\boldsymbol{\Gamma}&=(\bA^{(1)T} \bA^{(1)}) * \ldots * (\bA^{(N)T} \bA^{(N)}),\\
\bar{\boldsymbol{\Gamma}}^{(n)}&=(\bA^{(1)T} \bA^{(1)}) * \ldots *
(\bA^{(n-1)T} \bA^{(n-1)}) * \nonumber \\ 
& \qquad (\bA^{(n+1)T} \bA^{(n+1)}) * \ldots * (\bA^{(N)T} \bA^{(N)}),
\label{eq:Gamma}
\end{align}
where $*$ means element-wise multiplication.

The first-order optimality equations are then given by
\begin{align}
\bG^{(n)}=0=-\bbT_{(n)} \bar{\boldsymbol{\Phi}}^{(n)}+\bA^{(n)}\, \bar{\boldsymbol{\Gamma}}^{(n)}, \qquad n=1,\ldots,N.
\label{eq:foopt}
\end{align}
They offer an easy way to describe the ALS method for CP optimization: an ALS iteration proceeds by sequentially solving for each of the factor matrices $\bA^{(1)}, \ldots , \bA^{(N)}$ using the corresponding optimality equation $\bG^{(n)}=0$, updating the matrices $\bar{\boldsymbol{\Phi}}^{(n)}$ and $\bar{\boldsymbol{\Gamma}}^{(n)}$ along the way. As such, it is an example of a nonlinear block Gauss-Seidel method to solve the nonlinear equation system provided by the first-order optimality equations. The Matlab Tensor Toolbox \cite{KoldaTOOLBOX} implements efficient methods for computing the product $\bbT_{(n)} \bar{\boldsymbol{\Phi}}^{(n)}$ both for sparse and dense tensors.

One point that deserves special attention is the following: if no special care is taken during ALS iterations, it may happen that vectors in some modes diverge while others approach zero size, even if the process is converging to the desired CP decomposition; this is indeed possible due to the scaling indeterminacy. This type of anomalous behavior can be avoided by normalizing the columns of the factor matrices. After each complete ALS iteration, for each rank-one term we first normalize all factor vectors to size one, and then distribute the product of the normalization factors evenly to all factor vectors (using the $n$th root). We also order the rank-one terms in order of decreasing product of normalization factors. We apply this normalization and reordering after each ALS iteration, but also after each calculation of new accelerated or line search iterates in the N-GMRES optimization procedure. This controls possible erratic behavior in the order of the rank-one components and balances the relative sizes of the factor vectors as they are stacked in the iterate vectors $\bu_i$, which is expected to be beneficial for the convergence of the N-GMRES acceleration process. This consistent normalization thus appears to take care of the scaling and permutation indeterminacies present in the CP optimization problem, at least for the problems we have tested, as our numerical results confirm.

\subsection{Application of N-GMRES to CP Optimization: Further Particulars and Parameter Choices for Numerical Tests}
\label{subsec:particulars}
In this subsection we discuss some further particulars of our application of N-GMRES to CP optmization, and discuss some parameters for the numerical results to be presented in the next two sections.
As we mentioned before, and as in \cite{AcarCPOPT}, we utilize the Mor\'{e}-Thuente line search method
\cite{MoreThuente} as implemented in the Poblano toolbox for Matlab \cite{POBLANO}. For all experiments, the Mor\'{e}-Thuente line search parameters used were as follows: $10^{-4}$ for the function value tolerance, $10^{-2}$ for the gradient norm tolerance, a starting search step length of 1 and a maximum of 20 iterations. These values were also used for the N-CG tests in \cite{AcarCPOPT}, and we use them for our N-CG comparison runs as well. 
We use the N-CG variant with Polak-Ribi\`{e}re update formula \cite{Nocedal}.
As suggested in \cite{WashioNGMRES-ETNA}, the parameter $\delta$ in the term $\delta \, \bI$ added to normal equation matrix $\bP^T \, \bP$ in (\ref{eq:normal}), was chosen as $\epsilon$ times the size of the largest diagonal element of $\bP^T \, \bP$, with $\epsilon=10^{-12}$. We normally choose the N-GMRES window size $w$ equal to 20, which is confirmed to be a good choice in numerical tests described below. All initial guesses are determined uniformly randomly, as in \cite{ComonLS}, and when we compare different methods they are given the same random initial guess. All numerical tests were run on a laptop with a dual-core 2.53 GHz Intel Core i5 processor and 4GB of 1067 MHz DDR3 memory. Matlab version 7.11.0.584 (R2010b) 64-bit (maci64) was used for all tests.

\section{Numerical Results: Dense Tensor Test Problem}
\label{sec:NumDense}
\begin{figure}[!htbp]
  \centering
  \scalebox{0.5}{
  \includegraphics{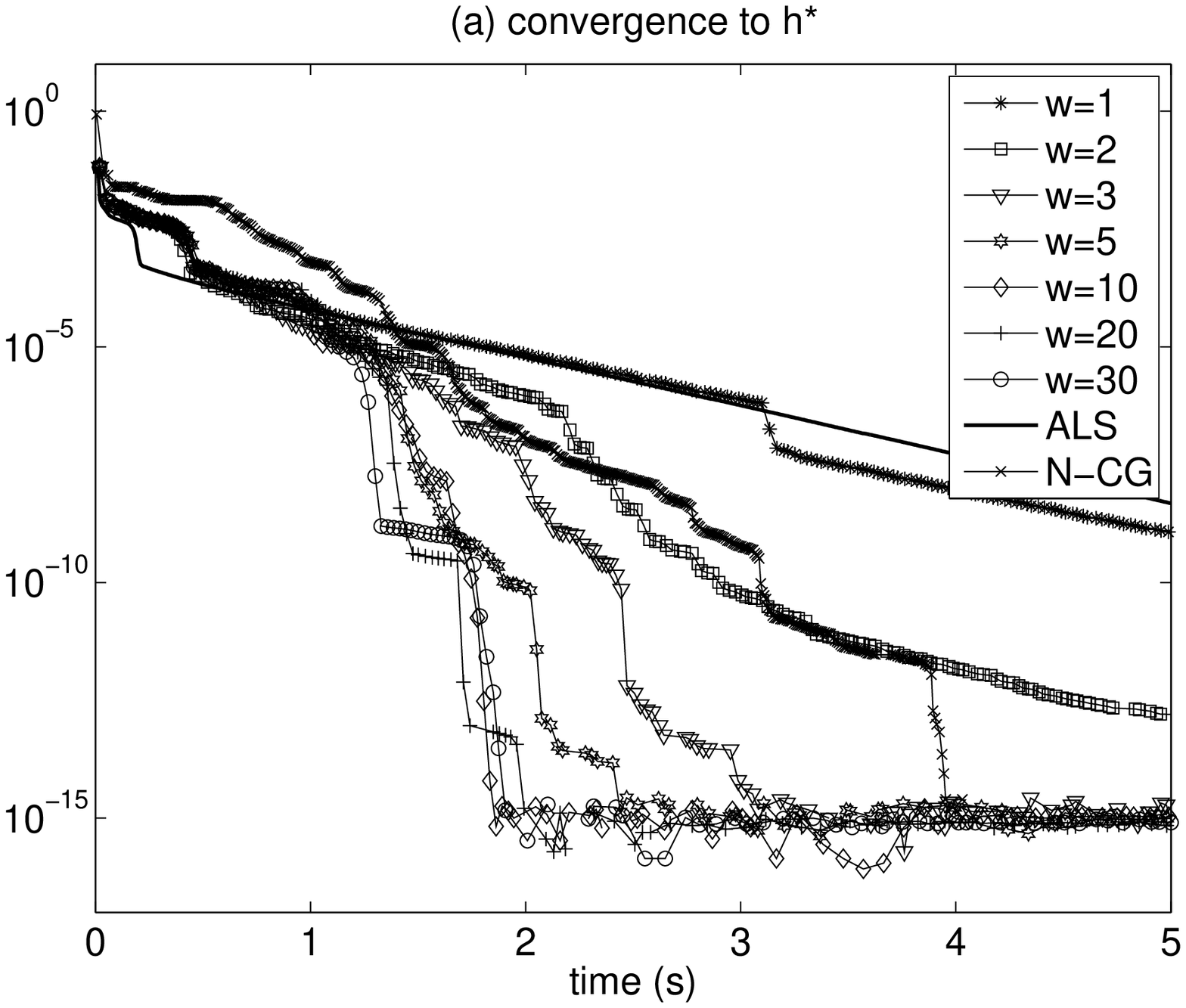}
  }
  \scalebox{0.5}{
  \includegraphics{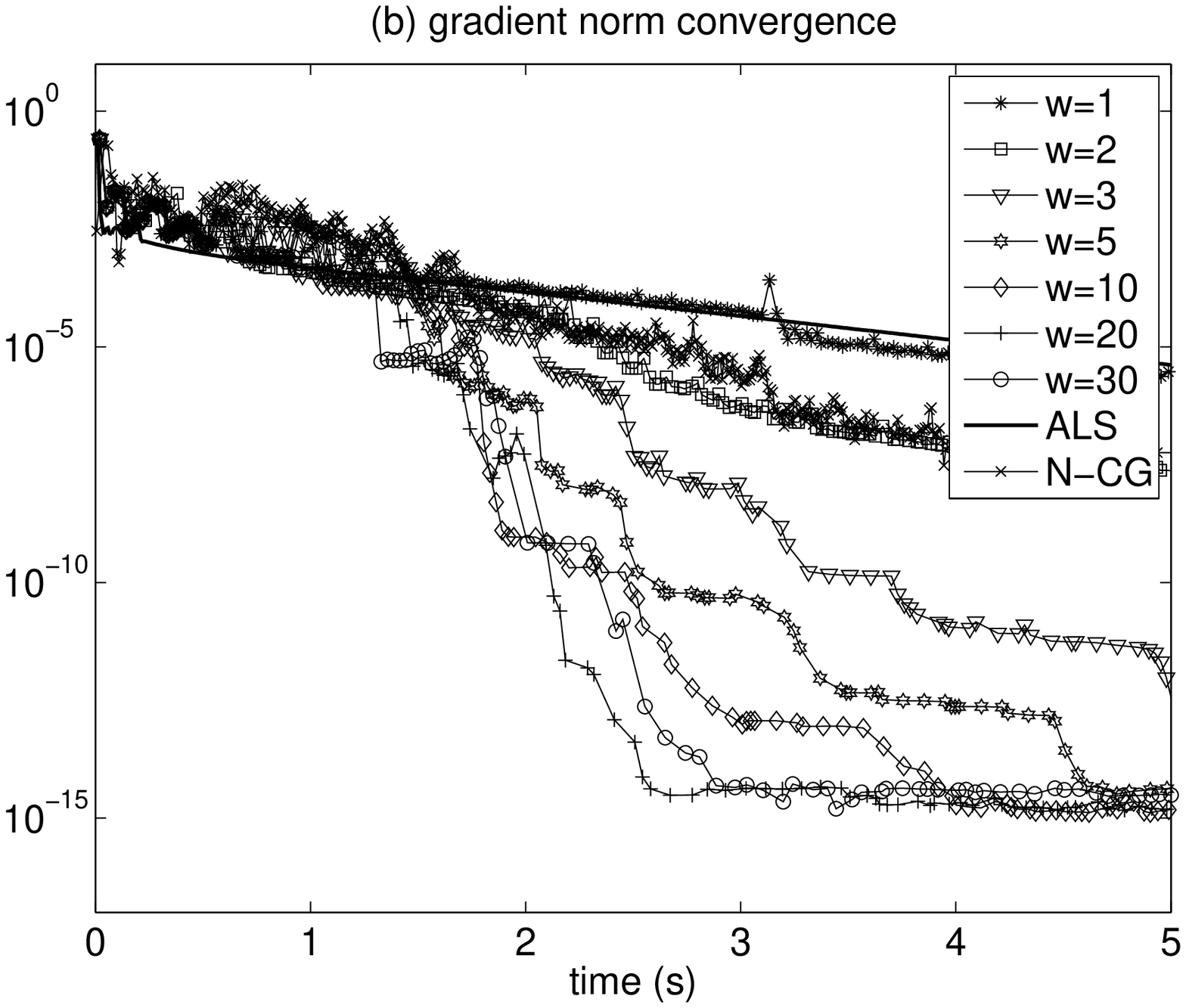}
  }
   \caption{Test Problem I with parameters ($s=50, c=0.9, R=3, l_{1}=1, l_{2}=1$). Convergence plots as a function of the N-GMRES window size, $w$. (a) Convergence of $|h(\mA_R^{(i)})-h^*|$, where $h^*$ is the value of the normalized distance measure, (\ref{eq:hA}), in the stationary point the method converges to. (b) Convergence of the normalized gradient of the objective function, $\|\bg(\mA_R^{(i)})\|_2$/$\|\mT\|_F$, indicating convergence to a stationary point. Window size $w=20$ emerges as a good choice for fast convergence when high accuracy is required.}
   \label{fig:denseDetTiming}
\end{figure}    

In this and the next section, we present extensive numerical tests for the N-GMRES optimization algorithm, compared with ALS and the N-CG algorithm of \cite{AcarCPOPT}, accessed via the Matlab Tensor Toolbox \cite{KoldaTOOLBOX}. In this section we present a class of dense test problems, and the next section deals with a sparse problem class.

Our dense test problem generates 3-way data tensors of various sizes starting from a canonical tensor with specified rank and random factor matrices that are modified to have pre-specified collinearity $c$, and noise is added. This is a standard CP test problem from \cite{TomasiPARAFAC}, and was also used in \cite{AcarCPOPT}. The collinearity between different columns of factor matrix $\bA^{(n)}$ is defined as
\begin{align}
c=\frac{\ba_r^{(n)T} \, \ba_s^{(n)}}{\|\ba_r^{(n)}\|_2 \, \|\ba_s^{(n)}\|_2}.
\label{eq:collin}
\end{align}
Collinearity close to 1 is known to lead to slow ALS convergence \cite{TomasiPARAFAC}.\\

\noindent
{\sc Test Problem I.} Generate a 3-way data tensor $\mT$ of noise-free rank $R$, size $s \times s \times s$, collinearity $c$, and noise levels $l_1$ and $l_2$ as follows. First generate an $R\times R$ matrix $\bK$ that has diagonal elements 1 and off-diagonal elements $c$, and compute the Cholesky factor $\bC$ of $\bK$. Then generate $n$ uniformly random $s \times R$ matrices, orthonormalize their columns using the QR decomposition, and multiply on the right with $C$. Then let $\mT_R$ be the canonical rank-$R$ tensor generated by these matrices as factor matrices. Two types of noise are added to $\mT_R$. Generate tensors $\mN_1$ and $\mN_2 \in \mathbb{R}^{s \times s \times s}$ with elements drawn from the standard normal distribution. An intermediate tensor $\hat{\mT}$ is generated as $\hat{\mT}=\mT_R+(100/l_1-1)^{-1/2} \, \|\mT_R\|_F \, \mN_1 / \|\mN_1\|_F$, and finally $\mT$ is obtained as $\mT=\hat{\mT}+(100/l_2-1)^{-1/2} \, \| \hat{\mT}\|_F \, (\mN_2 * \hat{\mT}) 
/ \| \mN_2 * \hat{\mT}  \|_F$, where * denotes element-wise multiplication.\\
  
We perform a series of tests computing a rank-$R$ CP decomposition of the tensors $\mT$, for various values of $s, c, R, l_{1}$ and $l_{2}$. Note that in this paper we do not study so-called overfactoring effects \cite{TomasiPARAFAC,AcarCPOPT}, in which numerical methods may produce approximations that do not give physically relevant information when the CP-rank $R$ is chosen larger than some rank intrinsic to the data tensor. Our reason for not considering overfactoring is that in this paper we accelerate ALS by the N-GMRES optimization approach, and we have observed that we almost always converge to the same stationary point as ALS. Since the overfactoring properties of ALS have been studied extensively elsewhere \cite{TomasiPARAFAC,AcarCPOPT}, we do not consider them here.

\begin{table}[h!]
    \begin{center}
        \begin{tabular}{|r|l|r|r|r|r|r|r|}\hline
\multicolumn{2}{|l|}{$h^*$ accuracy $10^{-3}$} & \multicolumn{2}{|c|}{ALS} & \multicolumn{2}{|c|}{N-GMRES} & \multicolumn{2}{|c|}{N-CG}\\
 \hline
 \multicolumn{2}{|c|}{problem parameters}  & it & time & it & time & it & time \\
 \hline
1 & $s=$20, $c=$0.5, $R=$3, $l_{1}=$1, $l_{2}=$1 & 18 & {\bf 0.083} & 16 & 0.21 & 34 & 0.17 \\
2 & $s=$20, $c=$0.5, $R=$5, $l_{1}=$10, $l_{2}=$5 & 9 & {\bf 0.083} & 8 & 0.17 & 64 & 0.51 \\
 \hline
3 & $s=$20, $c=$0.9, $R=$3, $l_{1}=$0, $l_{2}=$0 & 186 & 0.8 & 153 & 1.7 & 137 & {\bf 0.57} \\
4 & $s=$20, $c=$0.9, $R=$5, $l_{1}=$1, $l_{2}=$1 & 19 & {\bf 0.15} & 13 & 0.34 & 195 & 1.4 \\
 \hline
5 & $s=$50, $c=$0.5, $R=$3, $l_{1}=$1, $l_{2}=$1 & 11 & {\bf 0.089} & 8 & 0.21 & 38 & 0.46 \\
6 & $s=$50, $c=$0.5, $R=$5, $l_{1}=$10, $l_{2}=$5 & 10 & {\bf 0.15} & 9 & 0.3 & 50 & 0.97 \\
 \hline
7 & $s=$50, $c=$0.9, $R=$3, $l_{1}=$0, $l_{2}=$0 & 314 & 2.2 & 56 & {\bf 1.6} & 200 & 1.8 \\
8 & $s=$50, $c=$0.9, $R=$5, $l_{1}=$1, $l_{2}=$1 & 15 & {\bf 0.2} & 10 & 0.43 & $>$1821 & $>$32 \\
 \hline
9 & $s=$100, $c=$0.5, $R=$3, $l_{1}=$1, $l_{2}=$1 & 9 & {\bf 0.31} & 9 & 1.1 & 71 & 5.7 \\
10 & $s=$100, $c=$0.5, $R=$5, $l_{1}=$10, $l_{2}=$5 & 15 & {\bf 0.68} & 13 & 2.2 & 66 & 7.5 \\
 \hline
11 & $s=$100, $c=$0.9, $R=$3, $l_{1}=$0, $l_{2}=$0 & 178 & 5.9 & 30 & {\bf 3.9} & 340 & 23 \\
12 & $s=$100, $c=$0.9, $R=$5, $l_{1}=$1, $l_{2}=$1 & 12 & {\bf 0.52} & 9 & 1.7 & 260 & 24 \\
 \hline
        \end{tabular}
    \end{center}
    \caption{Test Problem I. Number of iterations and time (in seconds) until accuracy measure $|h(\mA_R^{(i)})-h^*|$ is reduced to $10^{-3}$. Here, $h^*$ is the value of the normalized distance measure, (\ref{eq:hA}), in the stationary point the methods converge to. The smallest times appear in bold. ALS is the fastest for most of these low-accuracy tests.}
    \label{tab:dense3}
\end{table}

Before performing tests for a wide range of parameter values $s, c, R, l_{1}$ and $l_{2}$, we look at a specific case and study the choice of N-GMRES window size $w$. Fig.~\ref{fig:denseDetTiming} shows convergence results for an instance of Test Problem I with parameters $s=$50, $c=$0.9, $R=$3, $l_{1}=$1 and $l_{2}=$1. This constitutes a `difficult' case with $c=0.9$, for which ALS converges rather slowly. It is the same case as in Figs.~\ref{fig:introHard}-\ref{fig:introHardTiming}. Fig.~\ref{fig:denseDetTiming} shows the effect of the window size $w$ on convergence speed. Several observations can be made. The choice $w=1$ does not converge much faster than ALS, but $w=2$ already leads to significant gains when high accuracy is desired, and $w=3$ already performs reasonably similar to the best choice. (Note again that these convergence plots take into account the extra costs of the N-GMRES optimization and line search in each iteration.) The optimal performance appears to occur, for this test problem, when $w\approx20$ is chosen. Fast convergence of $h(\mA_R^{(i)})$ to the optimal value $h^*$ (see (\ref{eq:hA})) is already fast obtained with high accuracy for $w\approx10$, but Fig.~\ref{fig:denseDetTiming}(b) shows that $w\approx20$ leads to faster convergence to a stationary point with high accuracy. Convergence plots as a function of iteration (not shown due to space constraints) show the same pattern, indicating that speed differences are almost entirely due to differing iteration counts, and not to differences in the amount of work to build and solve normal equation system (\ref{eq:normal}) for varying window size. We use window size $w=20$ in the numerical tests below.

\begin{table}[h!]
    \begin{center}
        \begin{tabular}{|r|l|r|r|r|r|r|r|}\hline
\multicolumn{2}{|l|}{$h^*$ accuracy $10^{-6}$} & \multicolumn{2}{|c|}{ALS} & \multicolumn{2}{|c|}{N-GMRES} & \multicolumn{2}{|c|}{N-CG}\\
 \hline
 \multicolumn{2}{|c|}{problem parameters}  & it & time & it & time & it & time \\
 \hline
1 & $s=$20, $c=$0.5, $R=$3, $l_{1}=$1, $l_{2}=$1 & 25 & {\bf 0.11} & 19 & 0.25 & 42 & 0.2 \\
2 & $s=$20, $c=$0.5, $R=$5, $l_{1}=$10, $l_{2}=$5 & 21 & {\bf 0.17} & 13 & 0.28 & 77 & 0.58 \\
 \hline
3 & $s=$20, $c=$0.9, $R=$3, $l_{1}=$0, $l_{2}=$0 & 949 & 4.1 & 167 & 1.9 & 197 & {\bf 0.79} \\
4 & $s=$20, $c=$0.9, $R=$5, $l_{1}=$1, $l_{2}=$1 & 582 & 4.2 & 109 & {\bf 3.6} & 614 & 3.9 \\
 \hline
5 & $s=$50, $c=$0.5, $R=$3, $l_{1}=$1, $l_{2}=$1 & 20 & {\bf 0.15} & 11 & 0.29 & 50 & 0.56 \\
6 & $s=$50, $c=$0.5, $R=$5, $l_{1}=$10, $l_{2}=$5 & 20 & {\bf 0.26} & 14 & 0.53 & 67 & 1.3 \\
 \hline
7 & $s=$50, $c=$0.9, $R=$3, $l_{1}=$0, $l_{2}=$0 & 1134 & 8 & 77 & {\bf 2.4} & 454 & 3.9 \\
8 & $s=$50, $c=$0.9, $R=$5, $l_{1}=$1, $l_{2}=$1 & 518 & {\bf 5.9} & 154 & 9.2 & $>$1821 & $>$32 \\
 \hline
9 & $s=$100, $c=$0.5, $R=$3, $l_{1}=$1, $l_{2}=$1 & 19 & {\bf 0.64} & 12 & 1.4 & 98 & 7.3 \\
10 & $s=$100, $c=$0.5, $R=$5, $l_{1}=$10, $l_{2}=$5 & 27 & {\bf 1.2} & 16 & 2.9 & 115 & 11 \\
 \hline
11 & $s=$100, $c=$0.9, $R=$3, $l_{1}=$0, $l_{2}=$0 & $>$800 & $>$27 & 69 & {\bf 8.4} & 673 & 44 \\
12 & $s=$100, $c=$0.9, $R=$5, $l_{1}=$1, $l_{2}=$1 & 457 & 19 & 85 & {\bf 19} & 620 & 52 \\
 \hline
        \end{tabular}
    \end{center}
    \caption{Test Problem I. Number of iterations and time (in seconds) until accuracy measure $|h(\mA_R^{(i)})-h^*|$ is reduced to $10^{-6}$. The smallest times appear in bold. For these medium-accuracy tests, ALS is still fastest for all `easy' cases ($c=0.5$), but N-GMRES or N-CG are faster for most of the `difficult' cases ($c=0.9$).}
    \label{tab:dense6}
\end{table}

Tables \ref{tab:dense3}-\ref{tab:dense10} show convergence results for a series of tests with a wide range of parameter values $s, c, R, l_{1}$ and $l_{2}$. Table \ref{tab:dense3} compares the number of iterations and times (in seconds) that the ALS, N-GMRES and N-CG (from \cite{AcarCPOPT}) methods need to reduce the accuracy measure $|h(\mA_R^{(i)})-h^*|$ to $10^{-3}$, where $h^*$ is the value of the normalized distance measure, (\ref{eq:hA}), in the stationary point the method converges to. All methods start from the same random initial guess, and converge to the same stationary point in these tests. This table compares the methods for a situation where low-accuracy results are considered sufficient. The smallest times appear in bold. Table \ref{tab:dense3} shows that ALS is normally the fastest when low accuracy is sufficient. This is consistent with what many others have observed for ALS-type algorithms, for CP decomposition and other problems, see, for example, 
\cite{TomasiPARAFAC,AcarCPOPT}. Note that the relative performance of the methods may be somewhat dependent on the initial guess, but this effect is partially averaged out by considering multiple tests with different parameters, and the results in the table give a representative overview of the relative performance of the methods.

\begin{table}[h!]
    \begin{center}
        \begin{tabular}{|r|l|r|r|r|r|r|r|}\hline
\multicolumn{2}{|l|}{$h^*$ accuracy $10^{-10}$} & \multicolumn{2}{|c|}{ALS} & \multicolumn{2}{|c|}{N-GMRES} & \multicolumn{2}{|c|}{N-CG}\\
 \hline
 \multicolumn{2}{|c|}{problem parameters}  & it & time & it & time & it & time \\
 \hline
1 & $s=$20, $c=$0.5, $R=$3, $l_{1}=$1, $l_{2}=$1 & 37 & {\bf 0.16} & 22 & 0.3 & 52 & 0.24 \\
2 & $s=$20, $c=$0.5, $R=$5, $l_{1}=$10, $l_{2}=$5 & 37 & {\bf 0.28} & 17 & 0.39 & 97 & 0.7 \\
 \hline
3 & $s=$20, $c=$0.9, $R=$3, $l_{1}=$0, $l_{2}=$0 & $>$1600 & $>$6.9 & 189 & {\bf 2.4} & $>$400 & $>$6.1 \\
4 & $s=$20, $c=$0.9, $R=$5, $l_{1}=$1, $l_{2}=$1 & $>$1200 & $>$8.6 & 139 & {\bf 4.5} & 1100 & 6.8 \\
 \hline
5 & $s=$50, $c=$0.5, $R=$3, $l_{1}=$1, $l_{2}=$1 & 32 & {\bf 0.23} & 16 & 0.42 & 67 & 0.69 \\
6 & $s=$50, $c=$0.5, $R=$5, $l_{1}=$10, $l_{2}=$5 & 36 & {\bf 0.44} & 17 & 0.67 & 89 & 1.6 \\
 \hline
7 & $s=$50, $c=$0.9, $R=$3, $l_{1}=$0, $l_{2}=$0 & $>$1200 & $>$8.5 & 104 & {\bf 3.5} & $>$553 & $>$7.6 \\
8 & $s=$50, $c=$0.9, $R=$5, $l_{1}=$1, $l_{2}=$1 & 1252 & 14 & 171 & {\bf 10} & $>$1821 & $>$32 \\
 \hline
9 & $s=$100, $c=$0.5, $R=$3, $l_{1}=$1, $l_{2}=$1 & 31 & {\bf 1} & 16 & 2 & 136 & 9.6 \\
10 & $s=$100, $c=$0.5, $R=$5, $l_{1}=$10, $l_{2}=$5 & 42 & {\bf 1.8} & 22 & 4.1 & 178 & 16 \\
 \hline
11 & $s=$100, $c=$0.9, $R=$3, $l_{1}=$0, $l_{2}=$0 & $>$800 & $>$27 & 99 & {\bf 17} & $>$748 & $>$60 \\
12 & $s=$100, $c=$0.9, $R=$5, $l_{1}=$1, $l_{2}=$1 & 1218 & 51 & 112 & {\bf 26} & 880 & 72 \\
 \hline
        \end{tabular}
    \end{center}
    \caption{Test Problem I. Number of iterations and time (in seconds) until accuracy measure $|h(\mA_R^{(i)})-h^*|$ is reduced to $10^{-10}$. The smallest times appear in bold. For these high-accuracy tests, ALS is still fastest for all `easy' cases ($c=0.5$), but N-GMRES is substantially faster for all the `difficult' cases ($c=0.9$).}
    \label{tab:dense10}
\end{table}
\begin{figure}[!htbp]
  \centering
  \scalebox{0.35}{
  \includegraphics{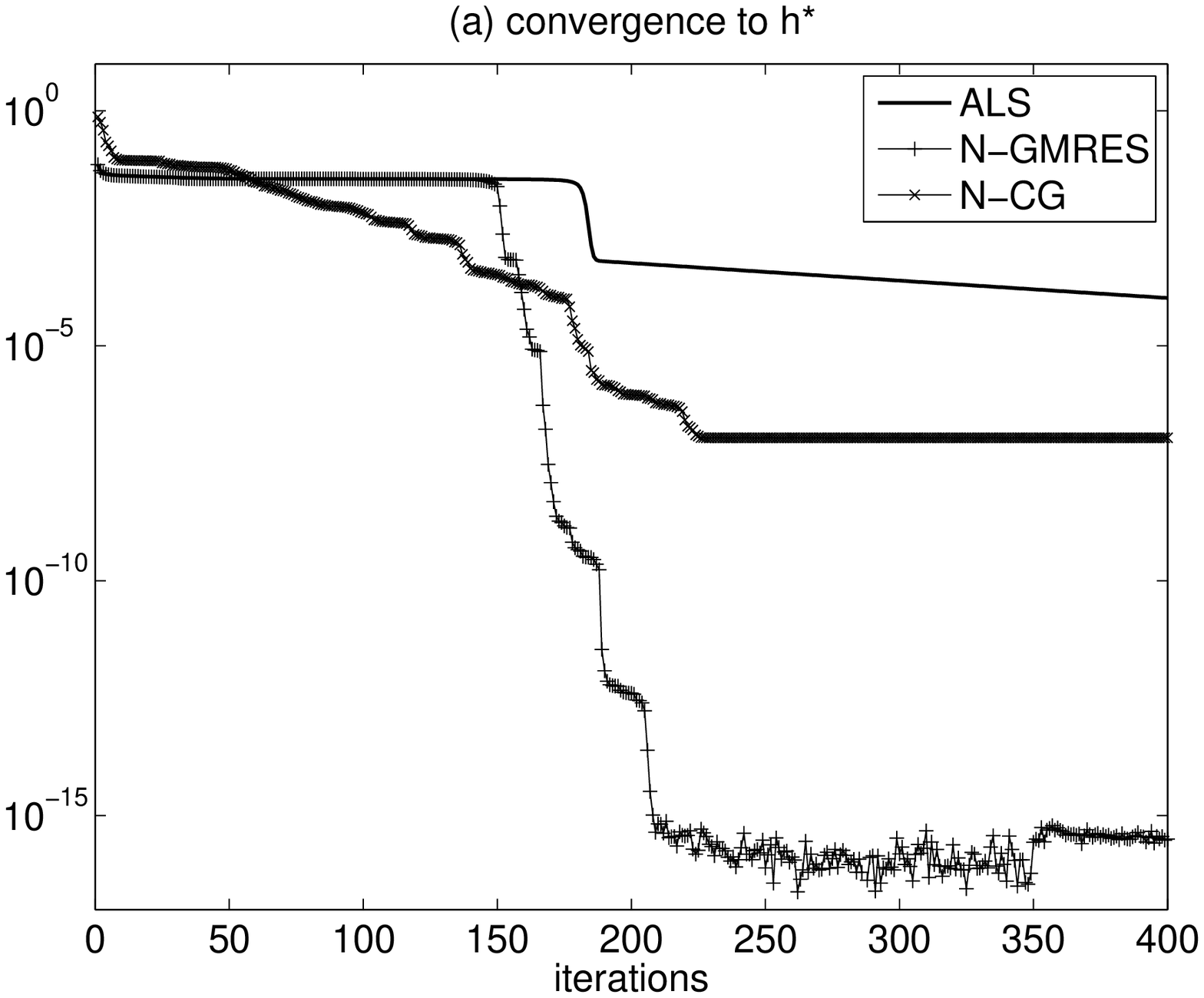}
  \includegraphics{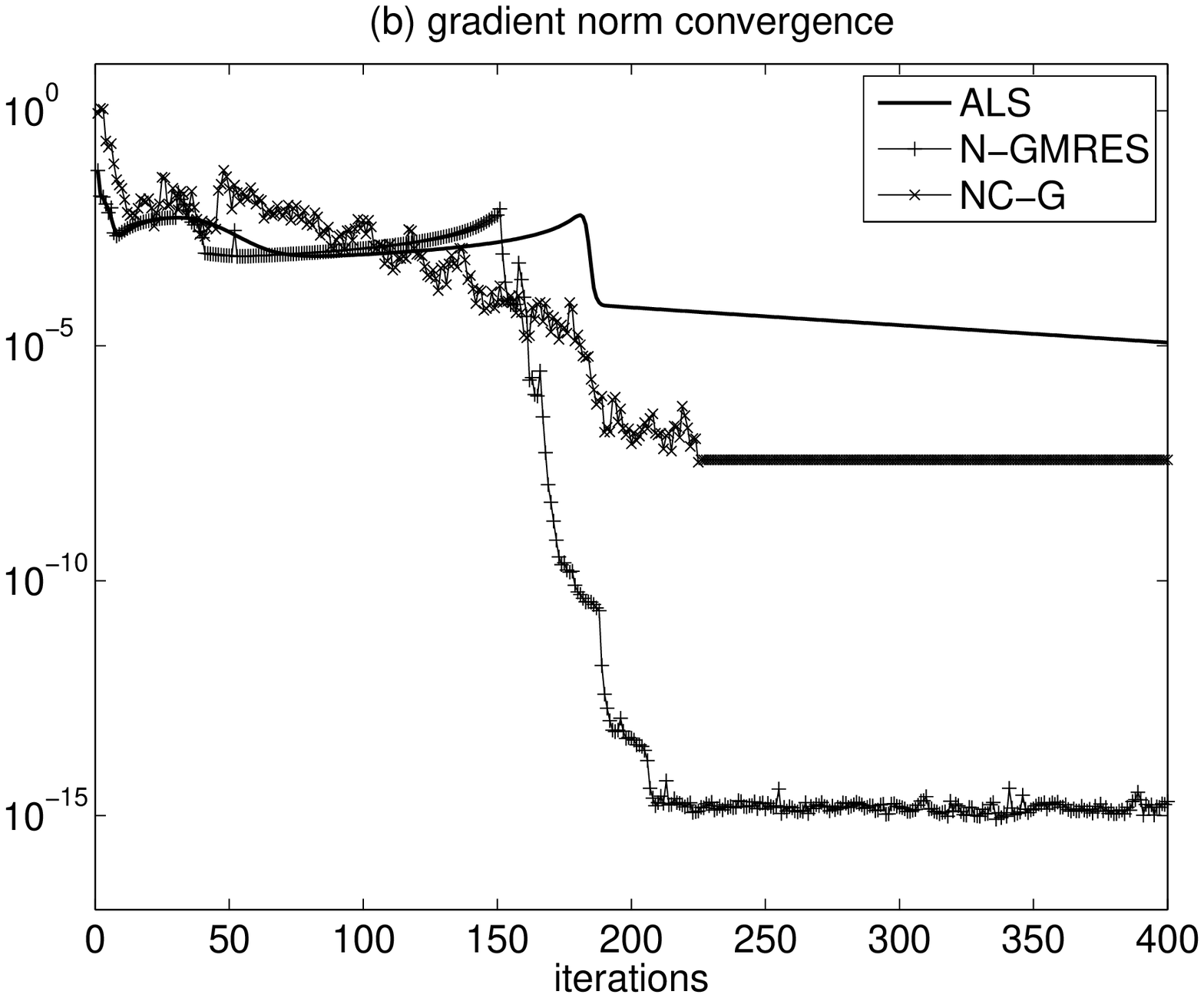}
  }
  \scalebox{0.35}{
  \includegraphics{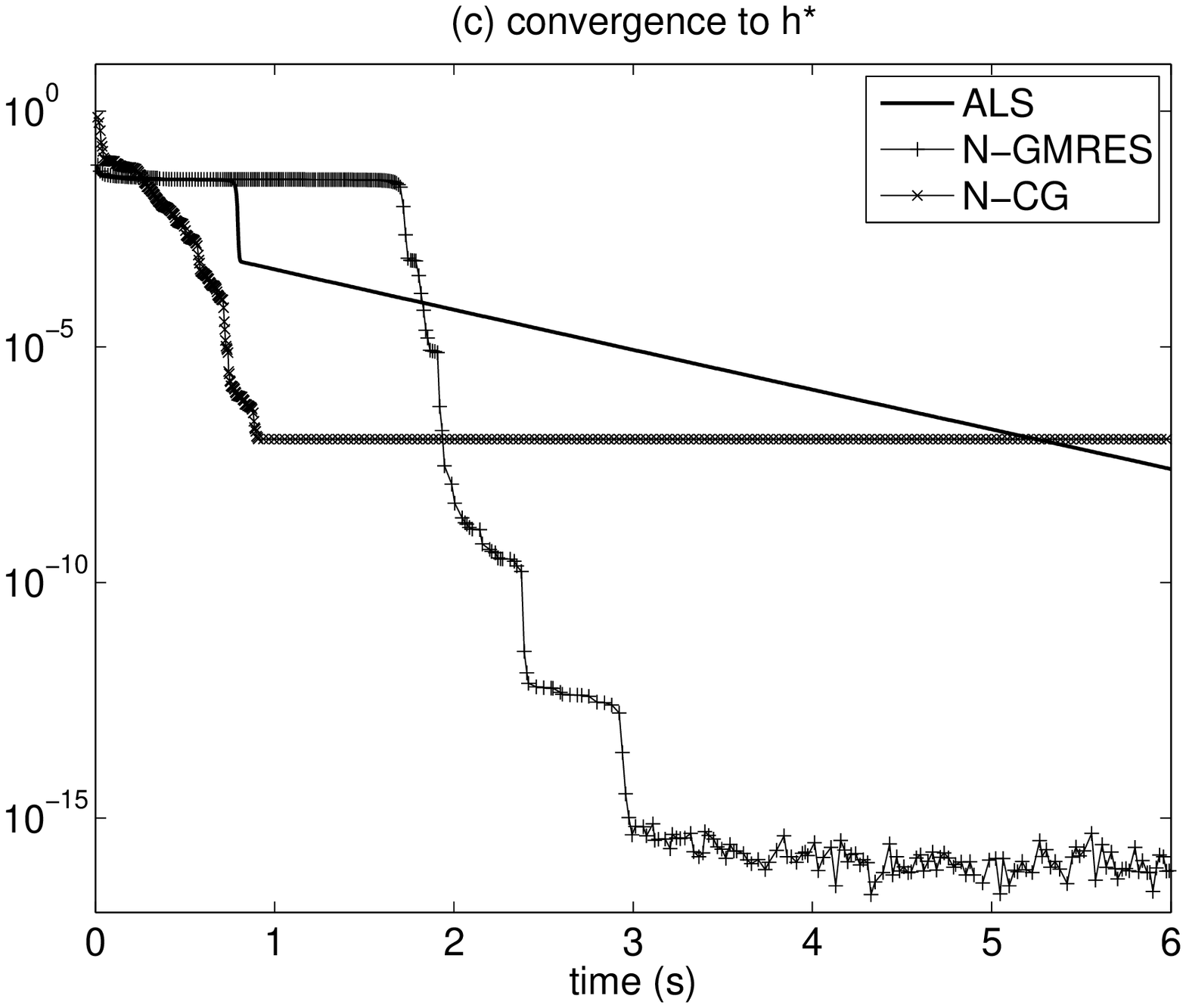}
  \includegraphics{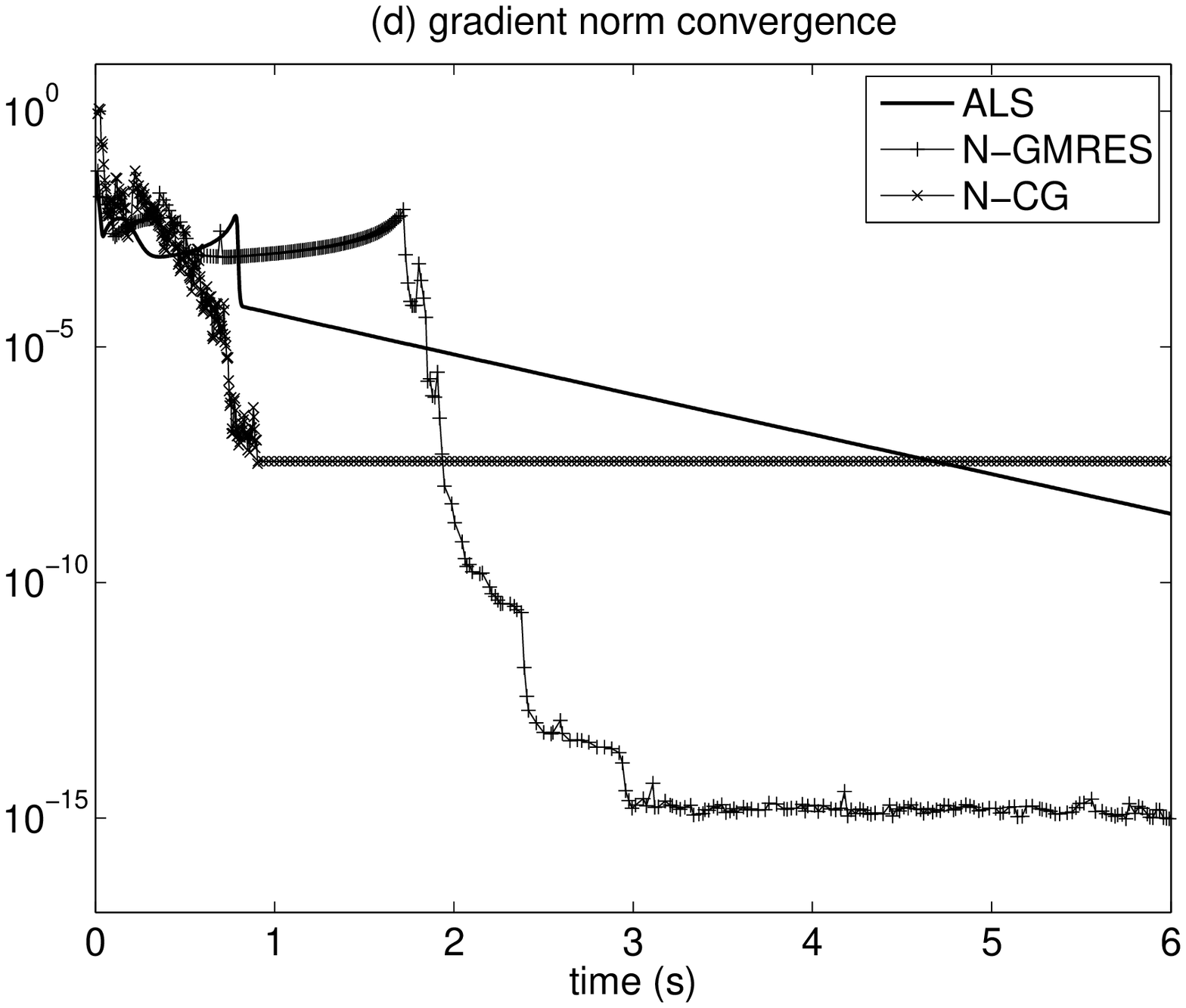}
  }
   \caption{Test Problem I. Convergence plots for case 3 from Tables  \ref{tab:dense3}-\ref{tab:dense10}. Panel (b) shows that N-GMRES convergence kicks in only when ALS has gotten over a `bump' in the gradient size. N-CG performs better in this case.}
   \label{fig:denseCase3}
\end{figure}    

Tables \ref{tab:dense6} and \ref{tab:dense10} show convergence results when higher accuracy is required, with $|h(\mA_R^{(i)})-h^*|<10^{-6}$ and $|h(\mA_R^{(i)})-h^*|<10^{-10}$, respectively. Table \ref{tab:dense6}, for medium accuracy $10^{-6}$, shows that ALS is (not unexpectedly) still fastest for all `easy' cases ($c=0.5$), but N-GMRES or N-CG are faster for most of the `difficult' cases ($c=0.9$). Table \ref{tab:dense10}, for high accuracy $10^{-10}$, confirms that ALS is still fastest for all `easy' cases ($c=0.5$), but N-GMRES is substantially faster for all `difficult' cases ($c=0.9$). Note also that N-GMRES is faster than N-CG for all of the `difficult' cases, sometimes substantially.

It is instructive to consider convergence histories in some more detail for some of the test problems in Tables  \ref{tab:dense3}-\ref{tab:dense10}. Fig.~\ref{fig:denseCase3} shows convergence plots for case 3 from Tables \ref{tab:dense3}-\ref{tab:dense10}.
For this problem (and the random initial guess used), convergence of N-GMRES only sets in after more than 150 iterations. Fig.~\ref{fig:denseCase3}(b) gives some indication as to why this is the case. N-GMRES accelerates ALS, but ALS needs to get over a `bump' in the size of the gradient before it gets close enough to a stationary point for the N-GMRES acceleration to become effective. This points out a fundamental property of the N-GMRES acceleration procedure that is the topic of this paper: N-GMRES is not expected to offer a systematic improvement in the global convergence properties of the `preconditoning' process it accelerates. The only potential help with global convergence would come from the line search in Step III of the algorithm, but this will only be effective when the search direction jointly determined by ALS and N-GMRES is suitable, and, as we argued before, we can only expect this to be the case consistently when the process approaches a stationary point. So N-GMRES still mainly relies on the preconditioning process to guide convergence on the global scale, and kicks in when close enough to a stationary point. In the case of Fig.~\ref{fig:denseCase3}, N-CG appears to perform better in getting close to a stationary point. Indeed, N-CG works in a way that is fundamentally different from N-GMRES: N-CG does not accelerate an underlying one-step method, but follows its own strategy for exploring the global search space, with globalization provided by its line search. For the computation in Fig.~\ref{fig:denseCase3}, this leads to faster N-CG convergence.

However, we have observed that convergence behavior like in Fig.~\ref{fig:denseCase3}, while it may occur for some problem parameters and initial conditions, is not typical for Test Problem I with $c=0.9$. Rather, convergence behavior as in Fig.~\ref{fig:denseCase11}, in which N-GMRES converges kicks in sooner, is more common (this is also indicated by the timing results in Tables \ref{tab:dense3}-\ref{tab:dense10}). Fig.~\ref{fig:denseCase11} shows convergence plots for case 11 from Tables \ref{tab:dense3}-\ref{tab:dense10}. In this case, N-GMRES convergence kicks in early, and N-CG's strategy for global convergence does not appear successful. For high accuracy, N-GMRES is much faster in this case than both ALS and N-CG, and this is generally the case (for the `difficult' problems with $c=0.9$), as confirmed by the numbers in Table \ref{tab:dense10}. We recall that ALS convergence plots of an `easy' case with $c=0.5$ are given in Fig.~\ref{fig:introEasy}.

\begin{figure}[!htbp]
  \centering
  \scalebox{0.35}{
  \includegraphics{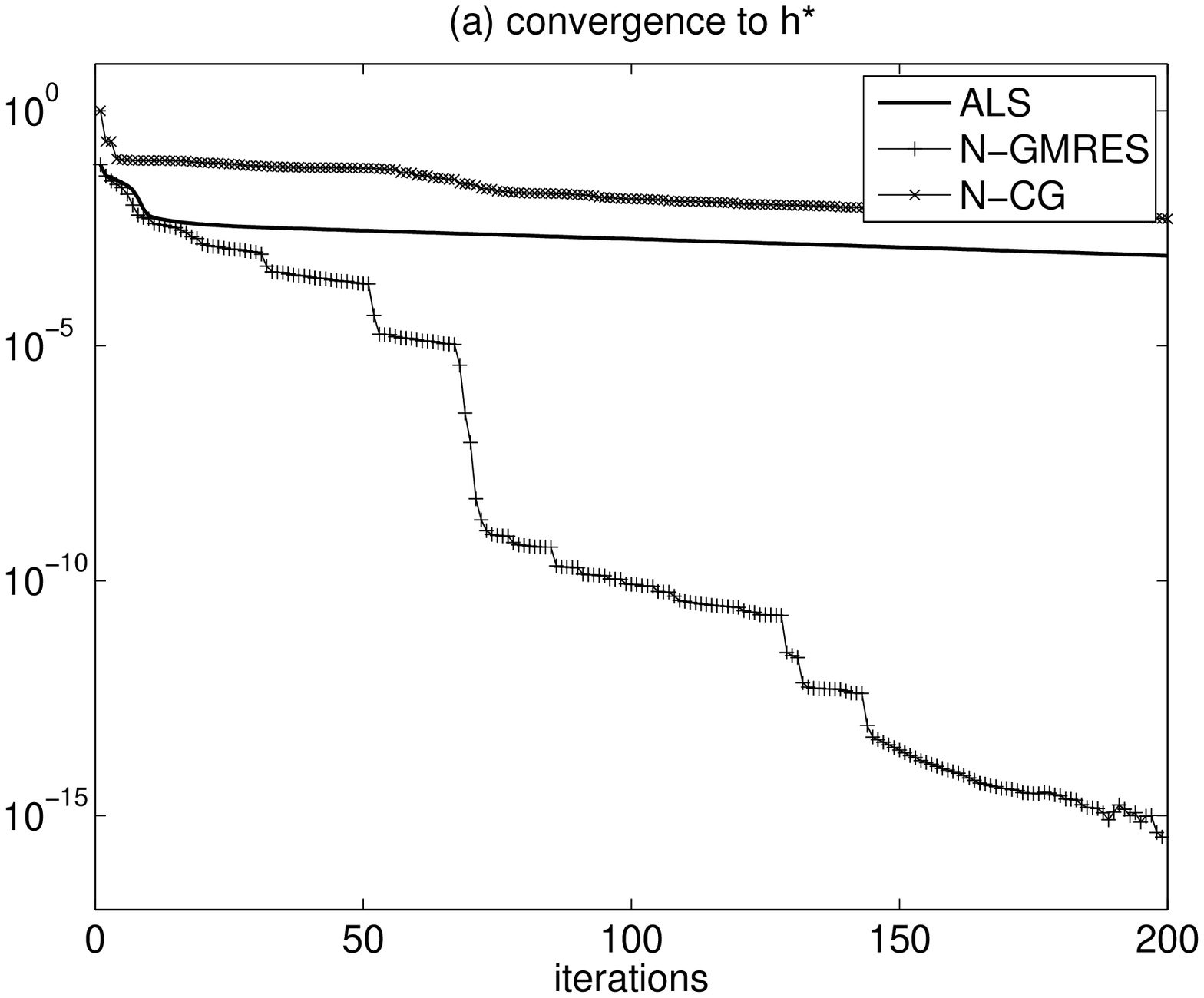}
  \includegraphics{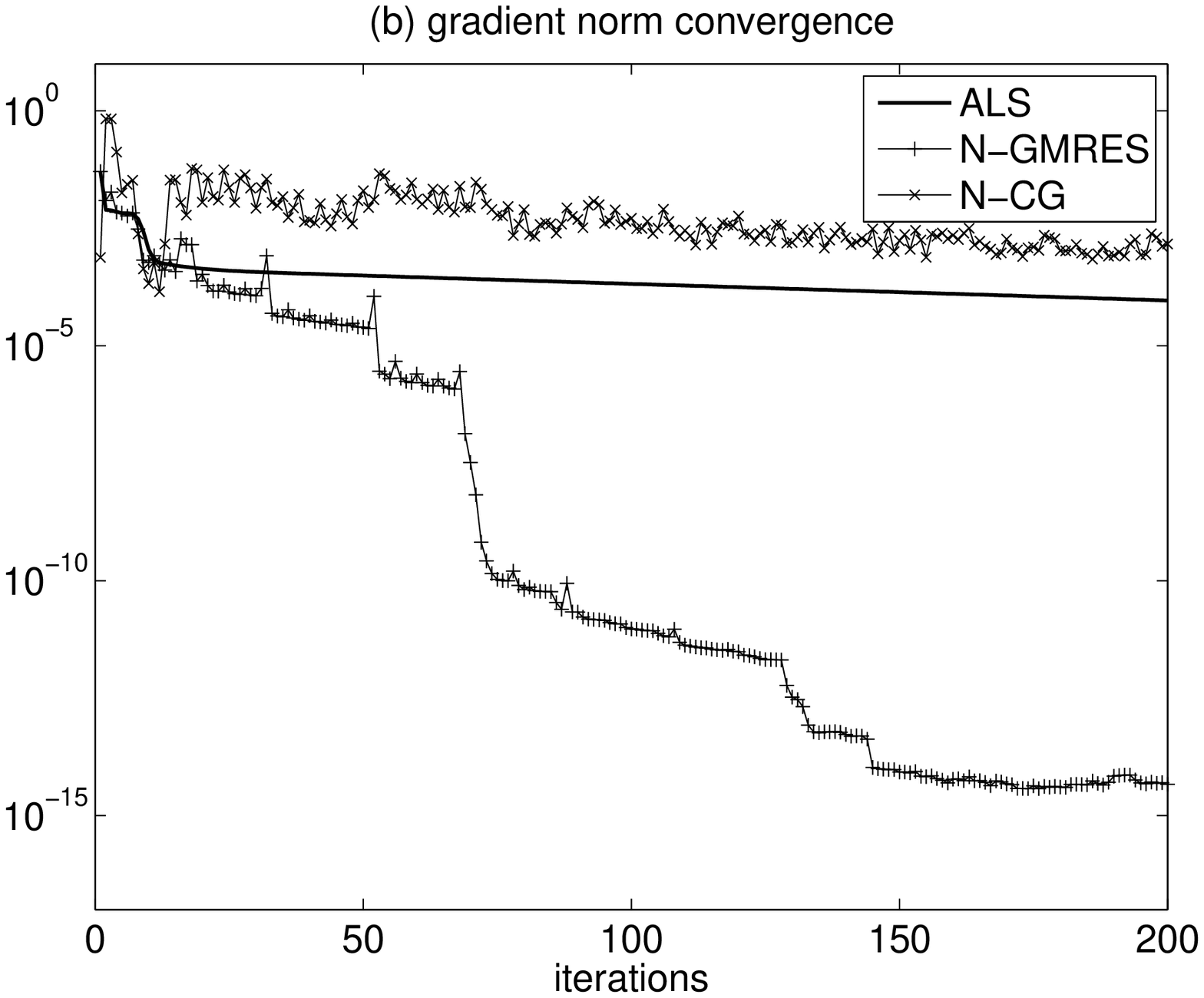}
  }
  \scalebox{0.35}{
  \includegraphics{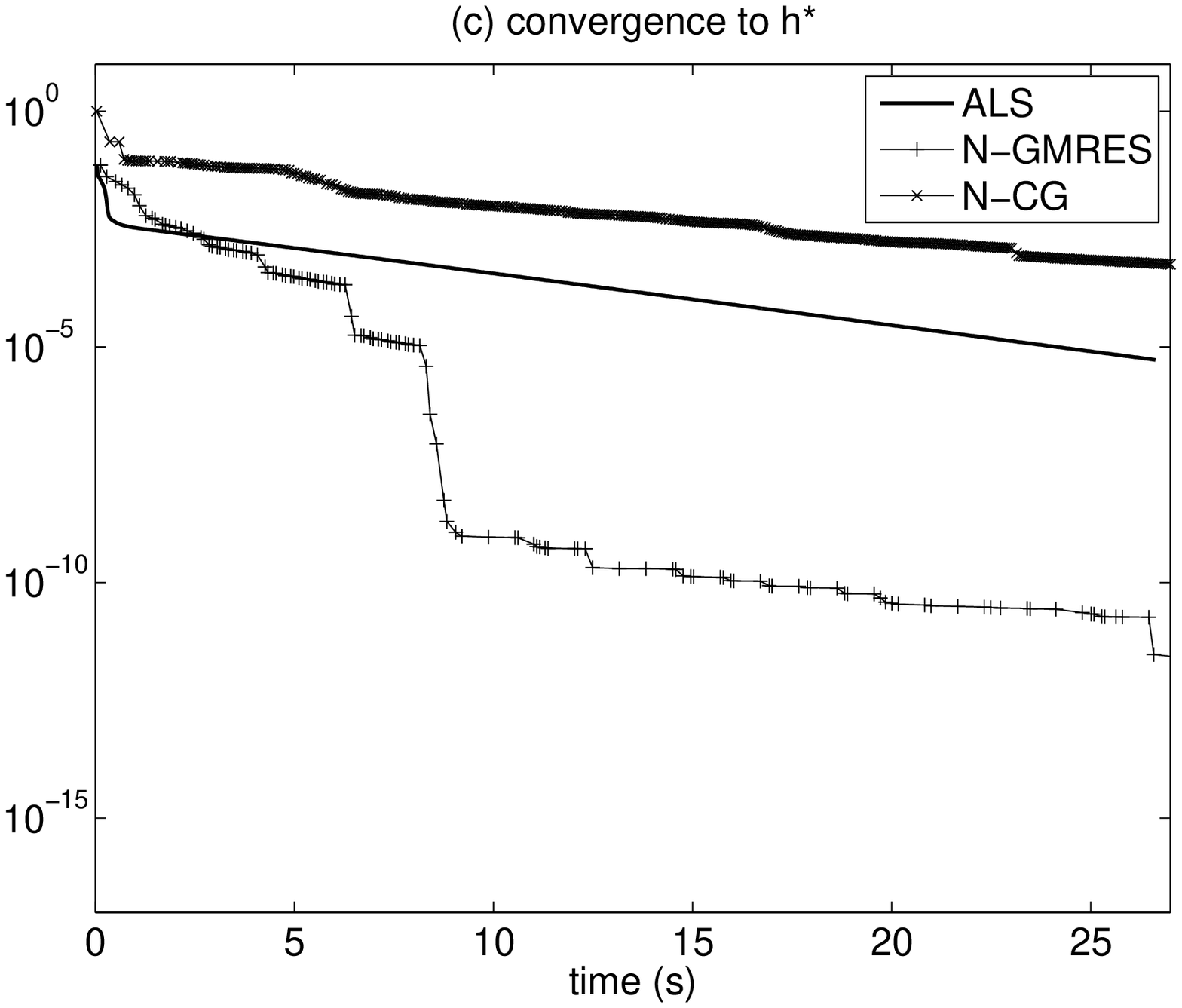}
  \includegraphics{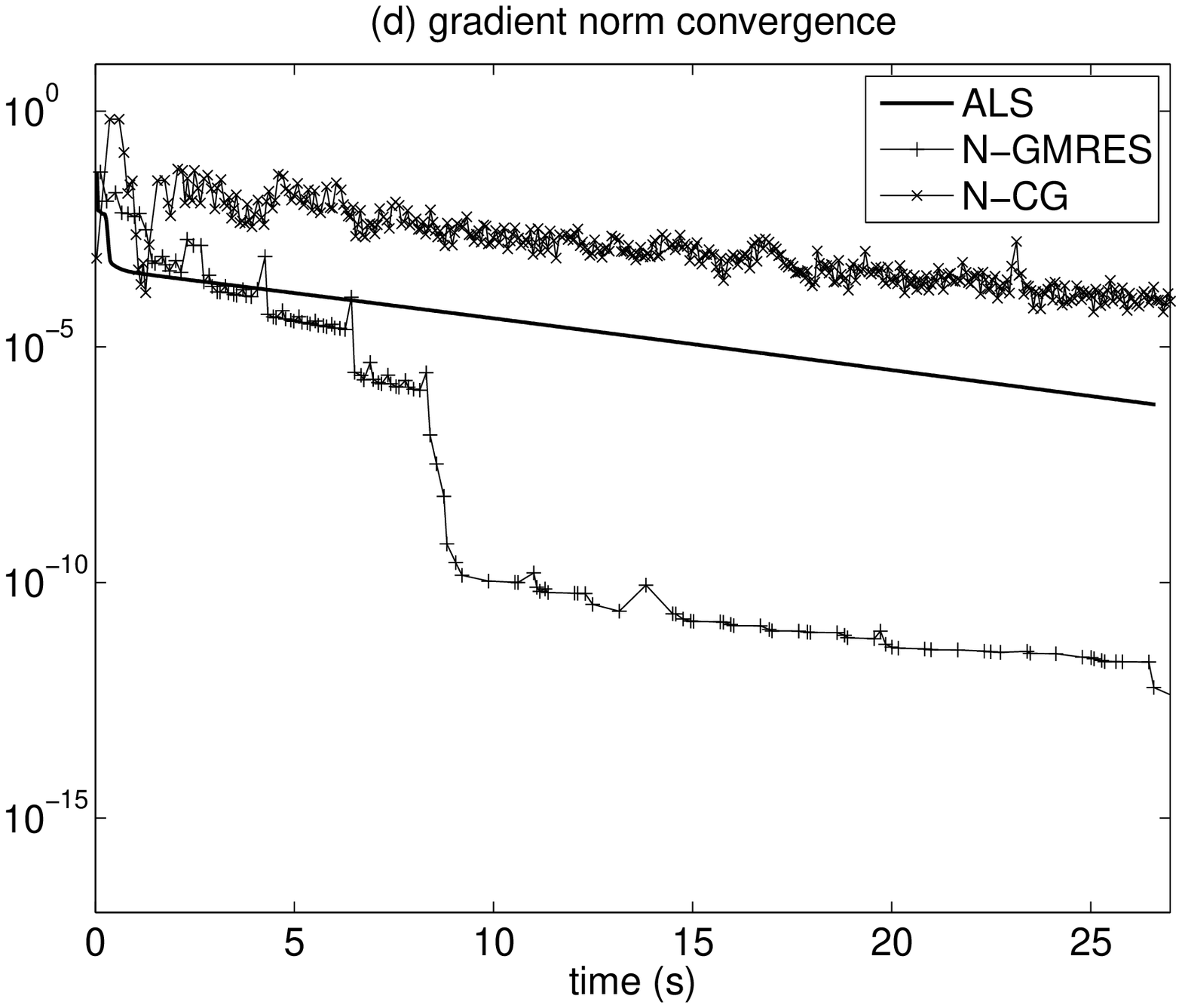}
  }
   \caption{Test Problem I. Convergence plots for case 11 from Tables \ref{tab:dense3}-\ref{tab:dense10}. N-GMRES convergence kicks in early, which leads to fast convergence, as is the case for most tests in Table \ref{tab:dense10} with $c=0.9$.}
   \label{fig:denseCase11}
\end{figure}    

Fig.~\ref{fig:denseCase3} shows another conspicuous difference in convergence behavior between N-GMRES and N-CG. 
N-CG convergence to a stationary point as measured by the size of the normalized gradient of the objective function (panel (b) in Fig.~\ref{fig:denseCase3}) stalls at a value of approximately $10^{-7}$. This is so because convergence of N-CG is limited by the accuracy of the line search, since it is the line search mechanism that ultimately determines the solution N-CG converges to. 
In principle this could be remedied by a more accurate line search, but this may incur extra cost, and we have found it difficult to significantly increase the accuracy of the Mor\'{e}-Thuente line search by changing its parameters. We have found that this gradient stalling occurs consistently for all problems we have considered; see also the other convergence plots in this paper.
In the test of Fig.~\ref{fig:denseCase3}, this convergence stalling also limits the accuracy of the stationary point to about $10^{-7}$, but this is not always the case (in some of the forthcoming plots N-CG converges to machine accuracy in terms of the $|h-h^*|$ measure).
It is interesting to note that N-GMRES does not appear to suffer from this stalling in the gradient norm convergence. The likely explanation of this goes as follows: the N-GMRES procedure is such that, close to a stationary point, the accelerated iterates provided by N-GMRES may converge efficiently to the stationary point, even without the help of the globalizing line search. This is so because the linearization of (\ref{eq:linearize}) is expected to be highly accurate close to a stationary point (for problems that are sufficiently smooth), and the accelerated iterate is highly accurate by itself. N-GMRES does thus not need to rely on the line search mechanism for accuracy once near a stationary point, in contrast to N-CG, which relies on the line search for accurate convergence. This difference points to a potential advantage of N-GMRES over N-CG.

\section{Numerical Results: Sparse Tensor Test Problem}
\label{sec:NumSparse}
\begin{figure}[!htbp]
  \centering
  \scalebox{0.35}{
  \includegraphics{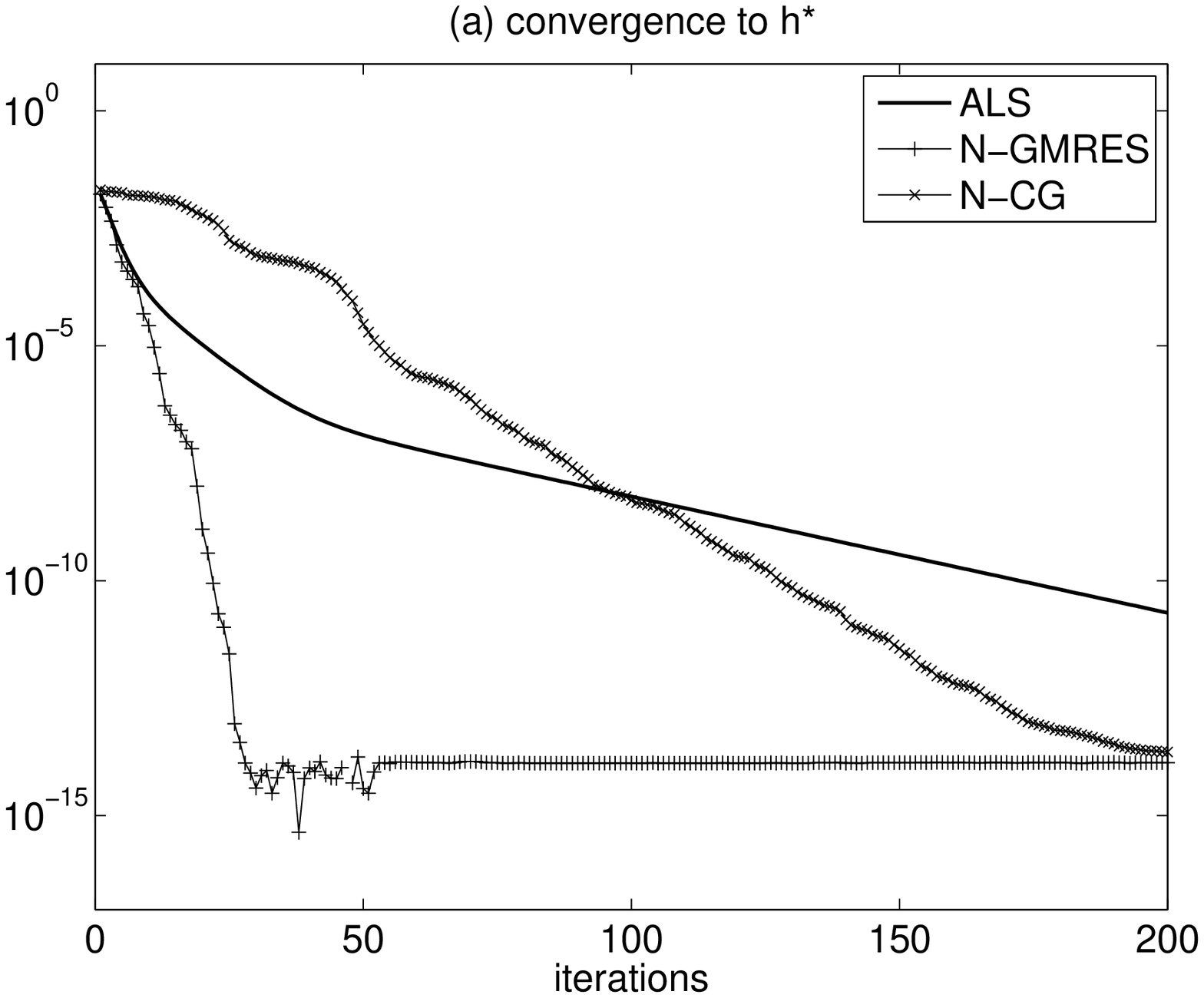}
  \includegraphics{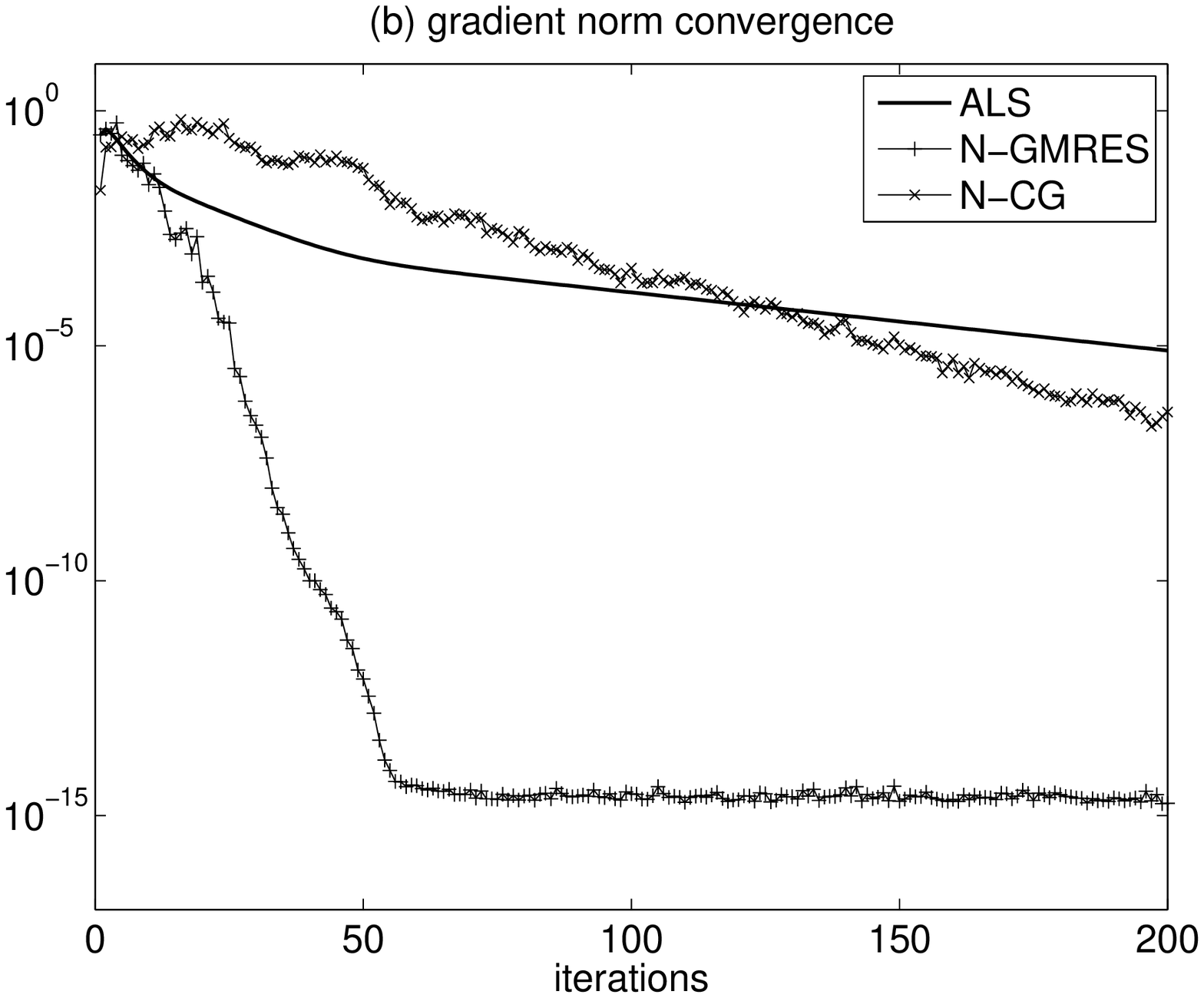}
  }
   \caption{Test Problem II. Parameters are $d=3$ (and thus $N=6$), $s=6$ and $R=3$, leading to a 6-way tensor of size $6 \times 6 \times 6 \times 6 \times 6 \times 6$ for which an $R=3$ CP approximation is sought. Convergence plots for ALS, N-CG, and N-GMRES with window size $w=20$.}
   \label{fig:sparseDetIterations}
\end{figure}    
\begin{figure}[!htbp]
  \centering
  \scalebox{0.5}{
  \includegraphics{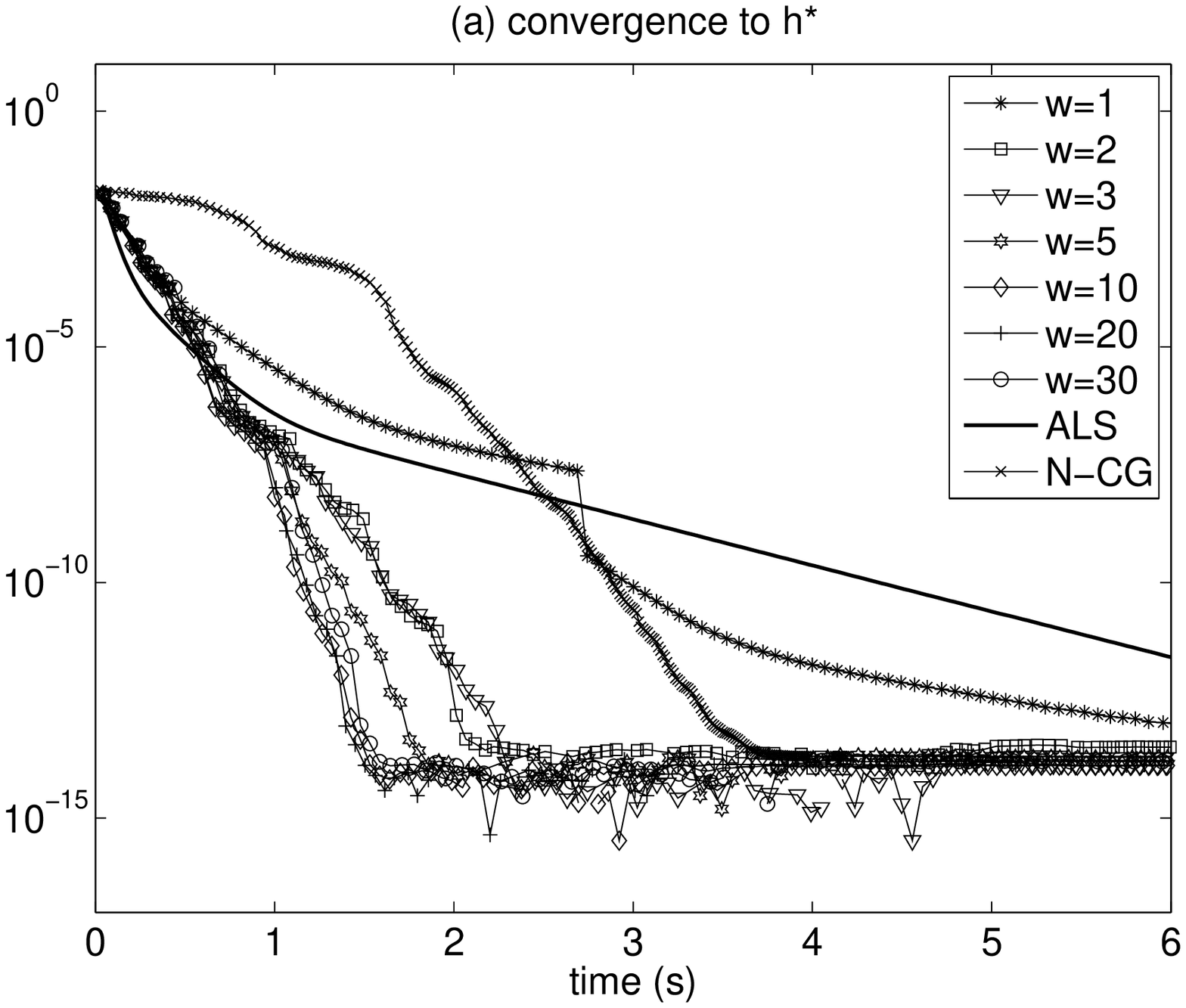}
  }
  \scalebox{0.5}{
  \includegraphics{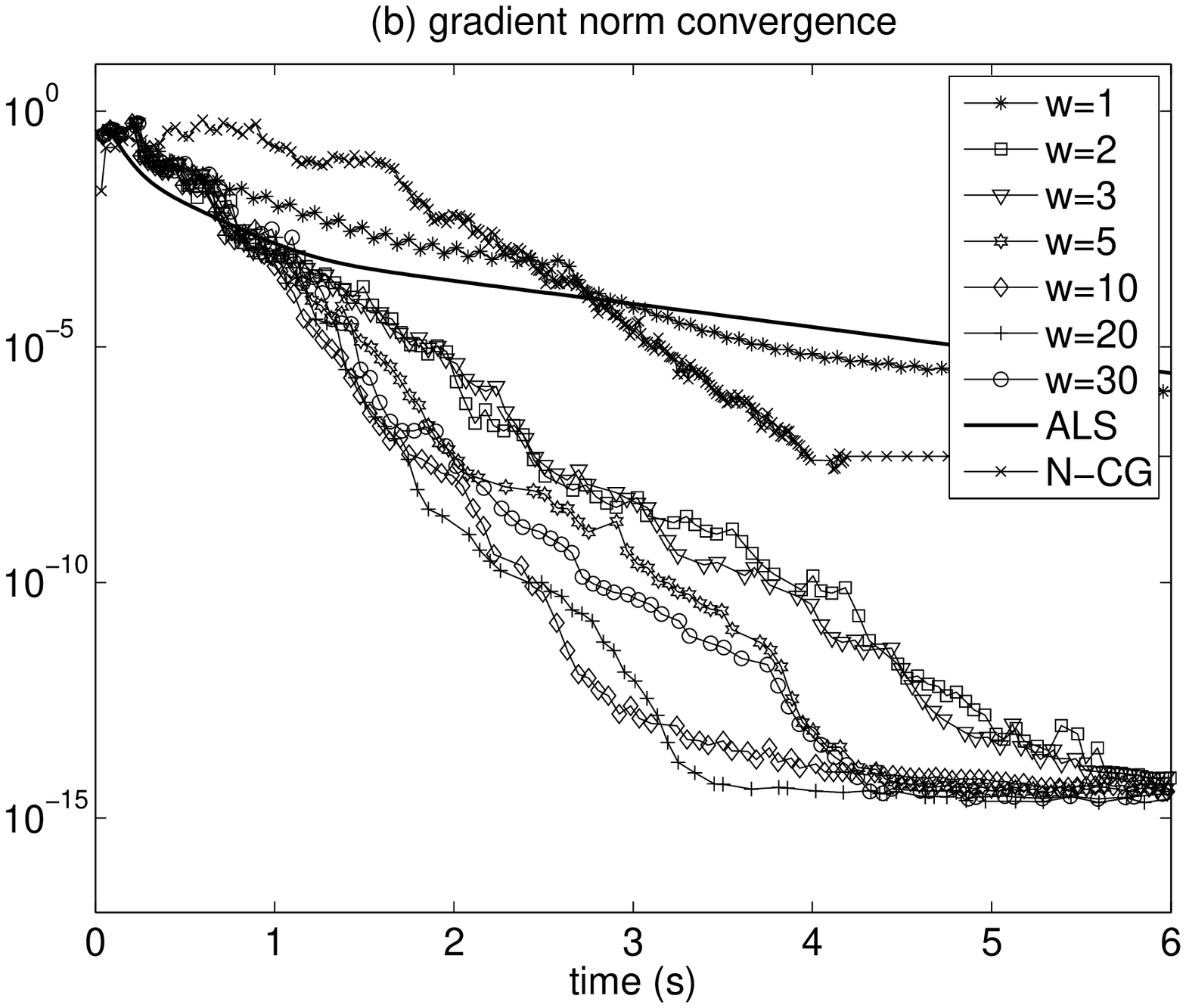}
  }
   \caption{Test Problem II ($N=6, s=6, R=3$).
Convergence plots as a function of the N-GMRES window size, $w$. Window size $w=20$ emerges as a good choice for fast convergence when high accuracy is required.}
   \label{fig:sparseDetTiming}
\end{figure}    
\begin{table}[h!]
    \begin{center}
        \begin{tabular}{|r|l|r|r|r|r|r|r|}\hline
\multicolumn{2}{|l|}{$h^*$ accuracy $10^{-3}$} & \multicolumn{2}{|c|}{ALS} & \multicolumn{2}{|c|}{N-GMRES} & \multicolumn{2}{|c|}{N-CG}\\
 \hline
\multicolumn{2}{|c|}{problem parameters}  & it & time & it & time & it & time \\
 \hline
1 & $N=$4, $s=$8, $R=$6 & 25 & {\bf 0.66} & 22 & 1.2 & 45 & 0.83 \\
2 & $N=$4, $s=$8, $R=$6 & 9 & {\bf 0.26} & 9 & 0.53 & 57 & 0.96 \\
3 & $N=$4, $s=$16, $R=$3 & 10 & {\bf 0.15} & 7 & 0.2 & 23 & 0.36 \\
4 & $N=$4, $s=$16, $R=$3 & 12 & {\bf 0.19} & 9 & 0.27 & 18 & 0.3 \\
 \hline
5 & $N=$6, $s=$4, $R=$2 & 12 & {\bf 0.22} & 9 & 0.29 & 30 & 0.53 \\
6 & $N=$6, $s=$4, $R=$2 & 9 & {\bf 0.17} & 7 & 0.25 & 325 & 2.7 \\
7 & $N=$6, $s=$8, $R=$5 & 9 & {\bf 0.37} & 9 & 1.1 & 100 & 24 \\
8 & $N=$6, $s=$8, $R=$5 & 8 & {\bf 0.37} & 8 & 1.6 & 86 & 22 \\
 \hline
9 & $N=$8, $s=$4, $R=$2 & 11 & {\bf 0.34} & 8 & 0.53 & 52 & 3.1 \\
10 & $N=$8, $s=$4, $R=$2 & 14 & {\bf 0.42} & 9 & 0.61 & 120 & 8.7 \\
 \hline
        \end{tabular}
    \end{center}
    \caption{Test Problem II. Number of iterations and time (in seconds) until accuracy measure $|h(\mA_R^{(i)})-h^*|$ is reduced to $10^{-3}$. The smallest times appear in bold. ALS is fastest for these low-accuracy tests.}
    \label{tab:sparse3}
\end{table}
\begin{table}[h!]
    \begin{center}
        \begin{tabular}{|r|l|r|r|r|r|r|r|}\hline
\multicolumn{2}{|l|}{$h^*$ accuracy $10^{-6}$} & \multicolumn{2}{|c|}{ALS} & \multicolumn{2}{|c|}{N-GMRES} & \multicolumn{2}{|c|}{N-CG}\\
 \hline
\multicolumn{2}{|c|}{problem parameters}  & it & time & it & time & it & time \\
 \hline
1 & $N=$4, $s=$8, $R=$6 & 182 & 4.4 & 44 & 2.5 & 91 & {\bf 1.4} \\
2 & $N=$4, $s=$8, $R=$6 & 67 & 1.6 & 20 & {\bf 1.2} & 163 & 2.2 \\
3 & $N=$4, $s=$16, $R=$3 & 410 & 5.9 & 94 & 3 & 103 & {\bf 1.3} \\
4 & $N=$4, $s=$16, $R=$3 & 203 & 3 & 53 & 1.7 & 124 & {\bf 1.4} \\
 \hline
5 & $N=$6, $s=$4, $R=$2 & 29 & 0.53 & 13 & {\bf 0.44} & 75 & 1 \\
6 & $N=$6, $s=$4, $R=$2 & 27 & 0.51 & 12 & {\bf 0.45} & 351 & 3 \\
7 & $N=$6, $s=$8, $R=$5 & 212 & {\bf 8.3} & 63 & 14 & 138 & 32 \\
8 & $N=$6, $s=$8, $R=$5 & 55 & {\bf 2.2} & 23 & 5.2 & 147 & 34 \\
 \hline
9 & $N=$8, $s=$4, $R=$2 & 38 & 1.1 & 15 & {\bf 1.1} & 75 & 4.3 \\
10 & $N=$8, $s=$4, $R=$2 & 42 & {\bf 1.3} & 19 & 1.4 & $>$280 & $>$19 \\
 \hline
        \end{tabular}
    \end{center}
    \caption{Test Problem II. Number of iterations and time (in seconds) until accuracy measure $|h(\mA_R^{(i)})-h^*|$ is reduced to $10^{-6}$. The smallest times appear in bold. N-GMRES and N-CG are competitive with ALS for these medium-accuracy tests.}
    \label{tab:sparse6}
\end{table}

\begin{table}[h!]
    \begin{center}
        \begin{tabular}{|r|l|r|r|r|r|r|r|}\hline
\multicolumn{2}{|l|}{$h^*$ accuracy $10^{-10}$} & \multicolumn{2}{|c|}{ALS} & \multicolumn{2}{|c|}{N-GMRES} & \multicolumn{2}{|c|}{N-CG}\\
 \hline
\multicolumn{2}{|c|}{problem parameters}  & it & time & it & time & it & time \\
 \hline
1 & $N=$4, $s=$8, $R=$6 & $>$400 & $>$9.6 & 55 & {\bf 3.1} & 380 & 3.7 \\
2 & $N=$4, $s=$8, $R=$6 & 242 & 5.8 & 26 & {\bf 1.5} & 327 & 3.5 \\
3 & $N=$4, $s=$16, $R=$3 & $>$800 & $>$12 & 119 & 3.8 & 419 & {\bf 3.5} \\
4 & $N=$4, $s=$16, $R=$3 & 724 & 11 & 84 & {\bf 2.7} & 375 & 3.2 \\
 \hline
5 & $N=$6, $s=$4, $R=$2 & 52 & 0.94 & 19 & {\bf 0.65} & 153 & 1.6 \\
6 & $N=$6, $s=$4, $R=$2 & 51 & 0.95 & 18 & {\bf 0.67} & 386 & 3.3 \\
7 & $N=$6, $s=$8, $R=$5 & 613 & 24 & 81 & {\bf 18} & 213 & 40 \\
8 & $N=$6, $s=$8, $R=$5 & 127 & {\bf 5.1} & 31 & 6.8 & 262 & 46 \\
 \hline
9 & $N=$8, $s=$4, $R=$2 & 70 & 2 & 21 & {\bf 1.5} & 111 & 5.2 \\
10 & $N=$8, $s=$4, $R=$2 & 72 & 2.1 & 24 & {\bf 1.8} & $>$280 & $>$19 \\
 \hline
        \end{tabular}
    \end{center}
    \caption{Test Problem II. Number of iterations and time (in seconds) until accuracy measure $|h(\mA_R^{(i)})-h^*|$ is reduced to $10^{-10}$. The smallest times appear in bold. N-GMRES outperforms ALS for most of these high-accuracy tests, and is normally also faster than N-CG.}
    \label{tab:sparse10}
\end{table}
\begin{figure}[!htbp]
  \centering
  \scalebox{0.35}{
  \includegraphics{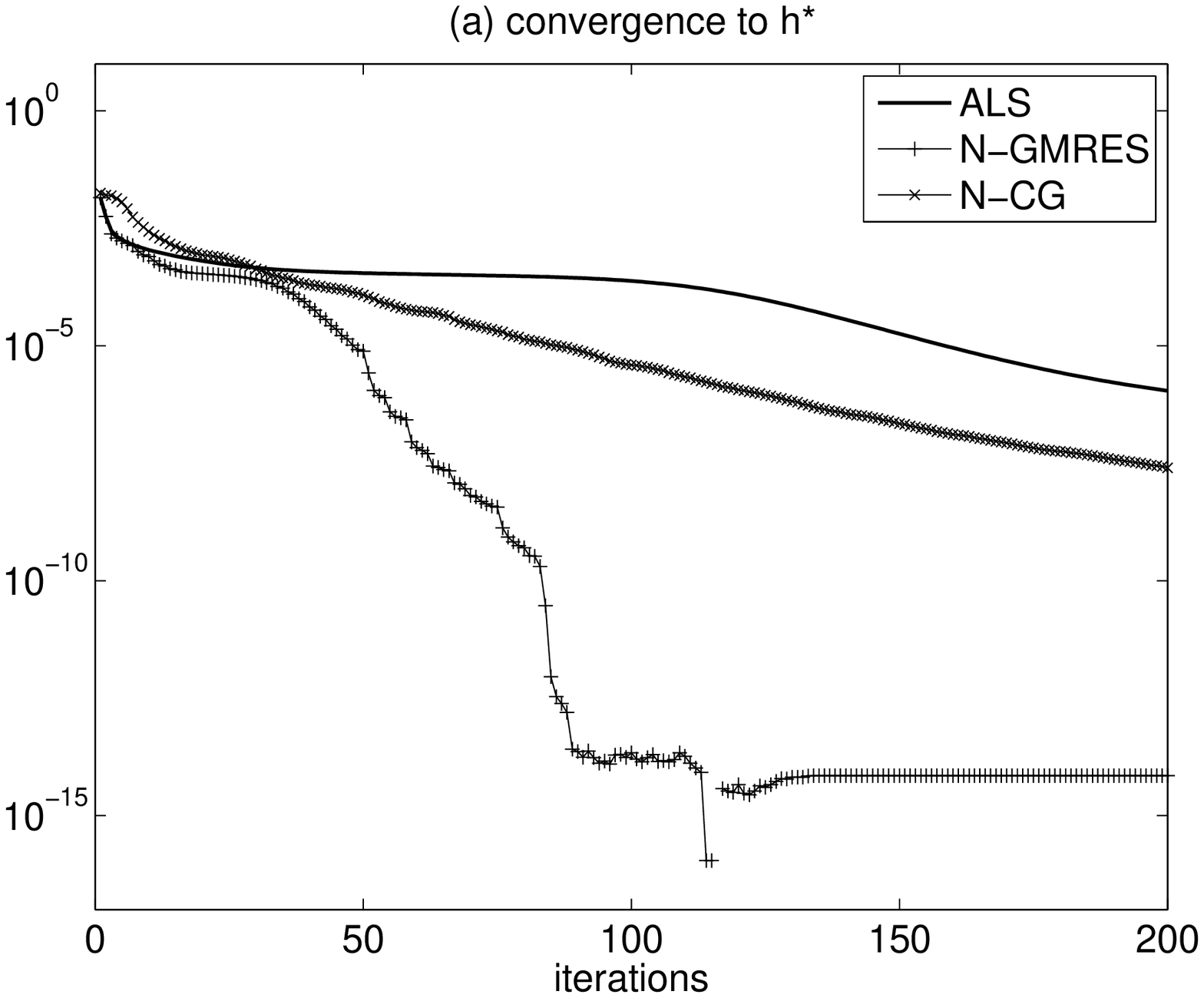}
  \includegraphics{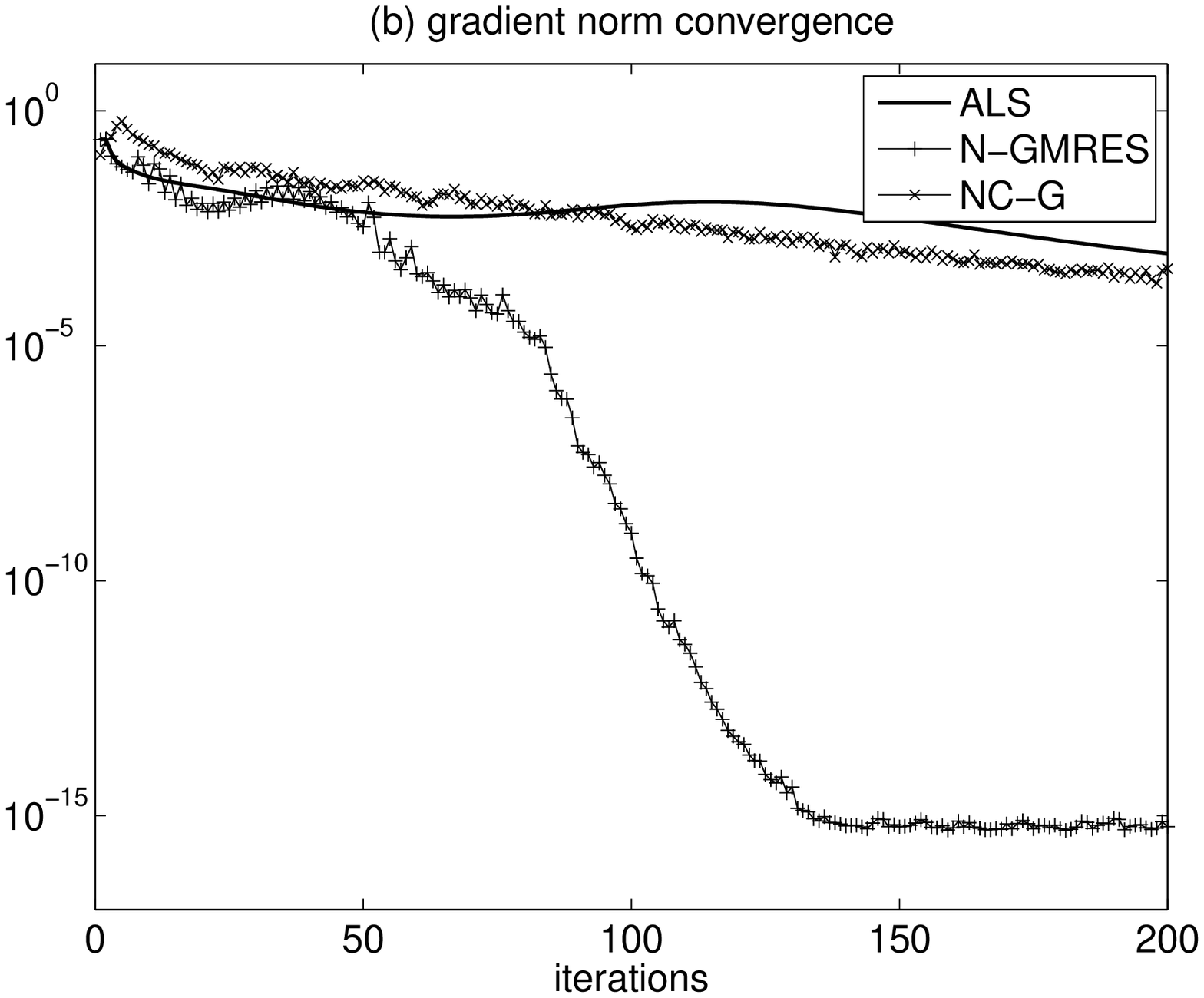}
  }
  \scalebox{0.35}{
  \includegraphics{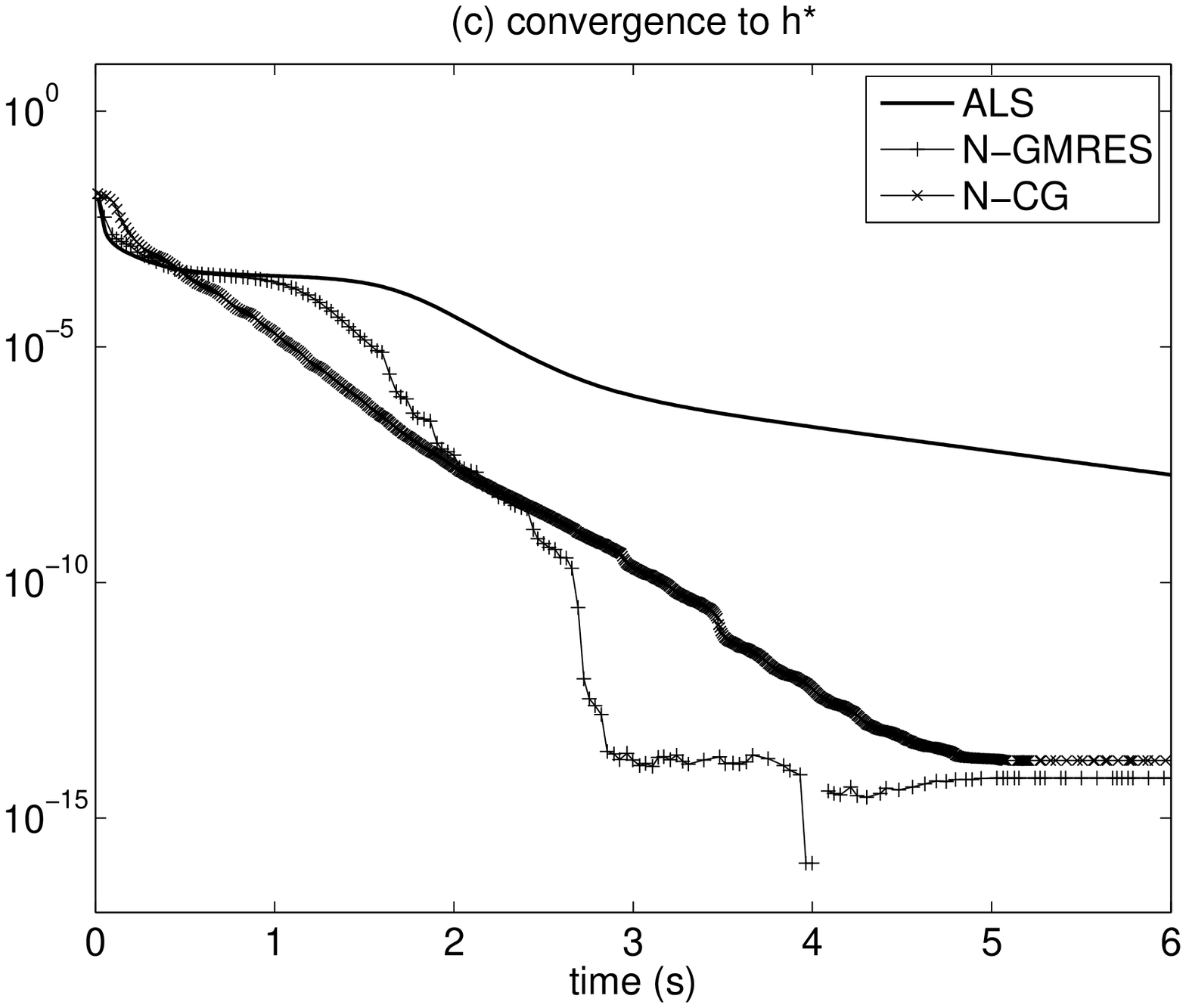}
  \includegraphics{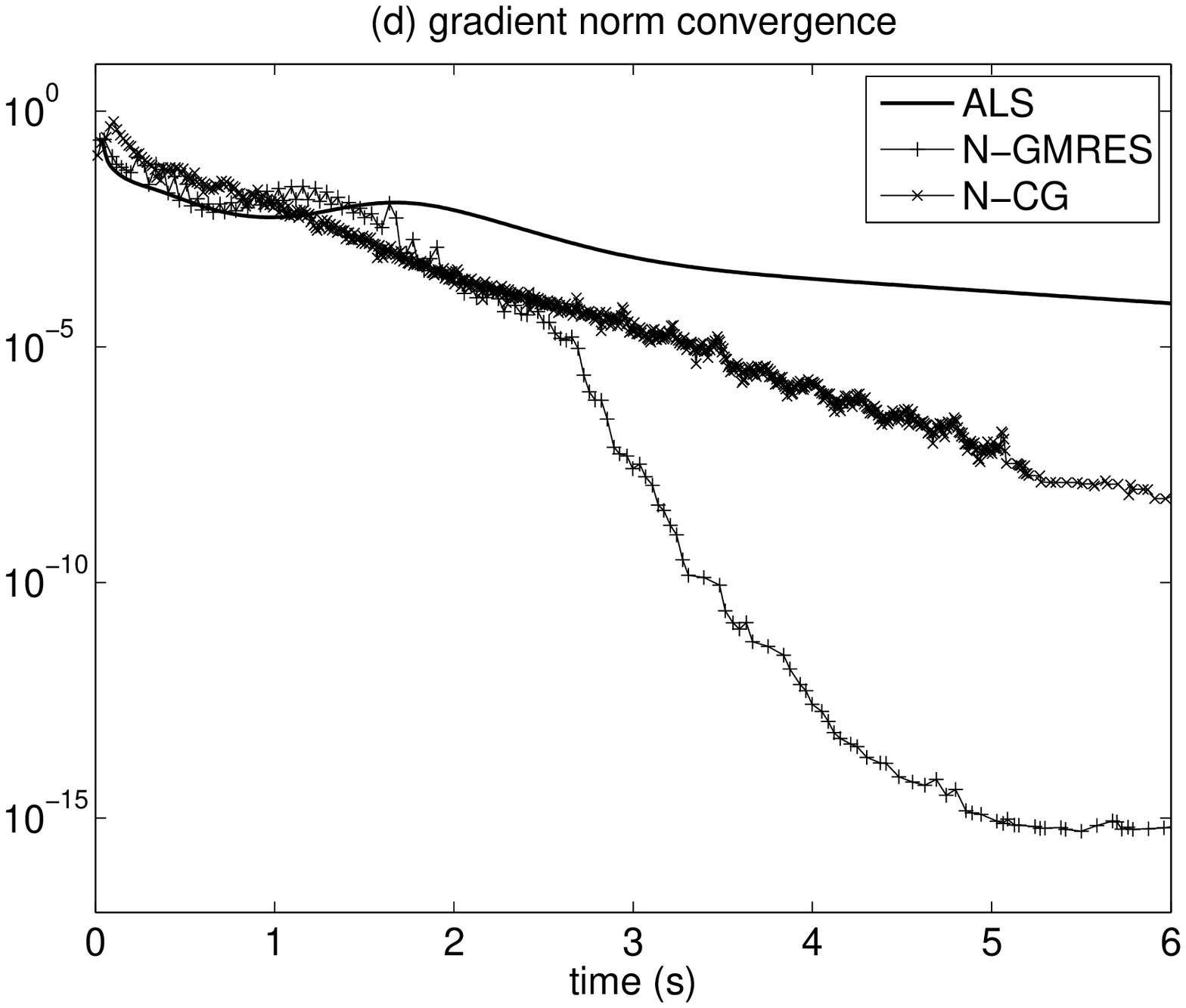}
  }
   \caption{Test Problem II. Convergence plots for case 4 from Tables \ref{tab:sparse3}-\ref{tab:sparse10}.}
   \label{fig:sparseCase4}
\end{figure}    

\begin{figure}[!htbp]
  \centering
  \scalebox{0.35}{
  \includegraphics{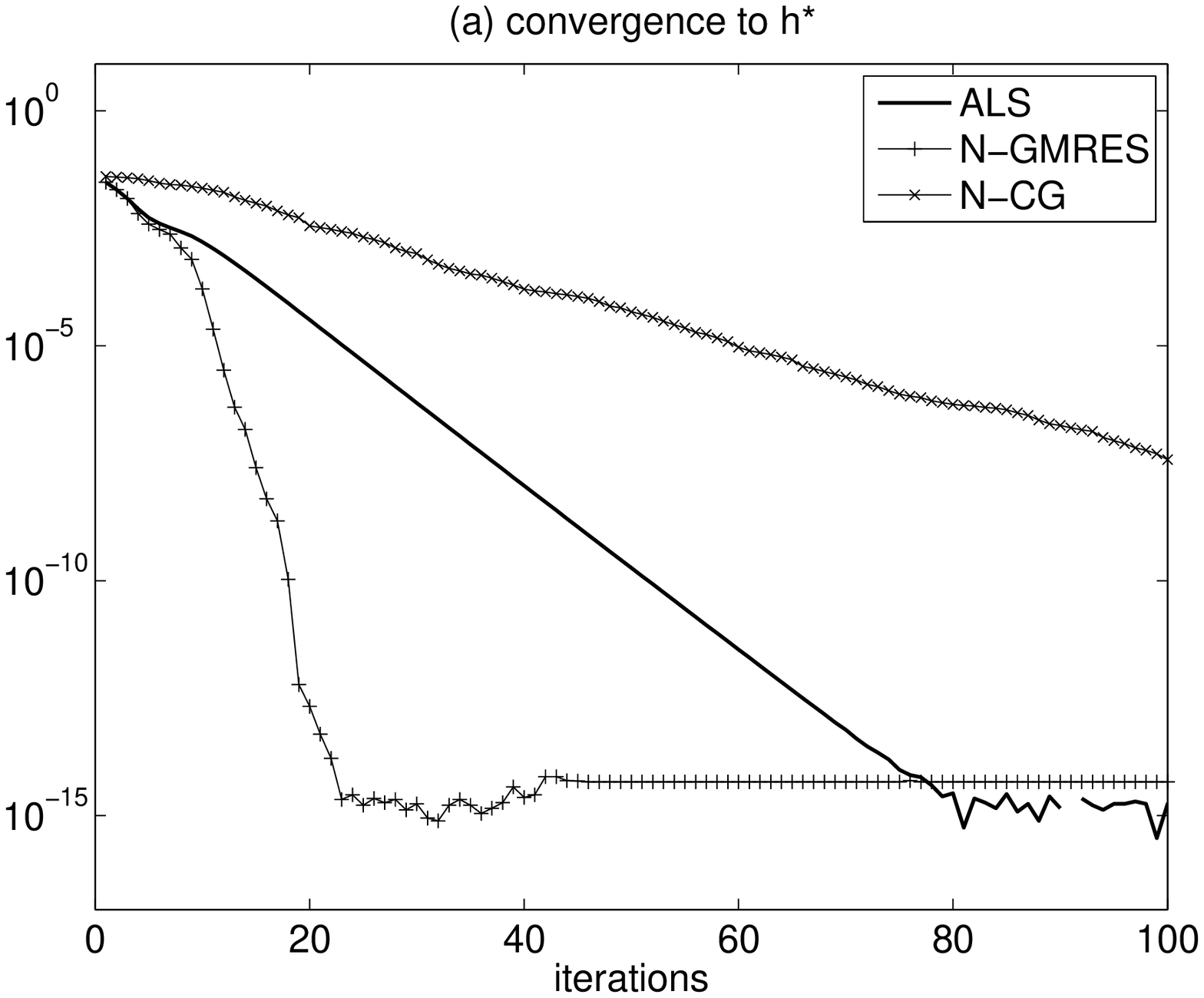}
  \includegraphics{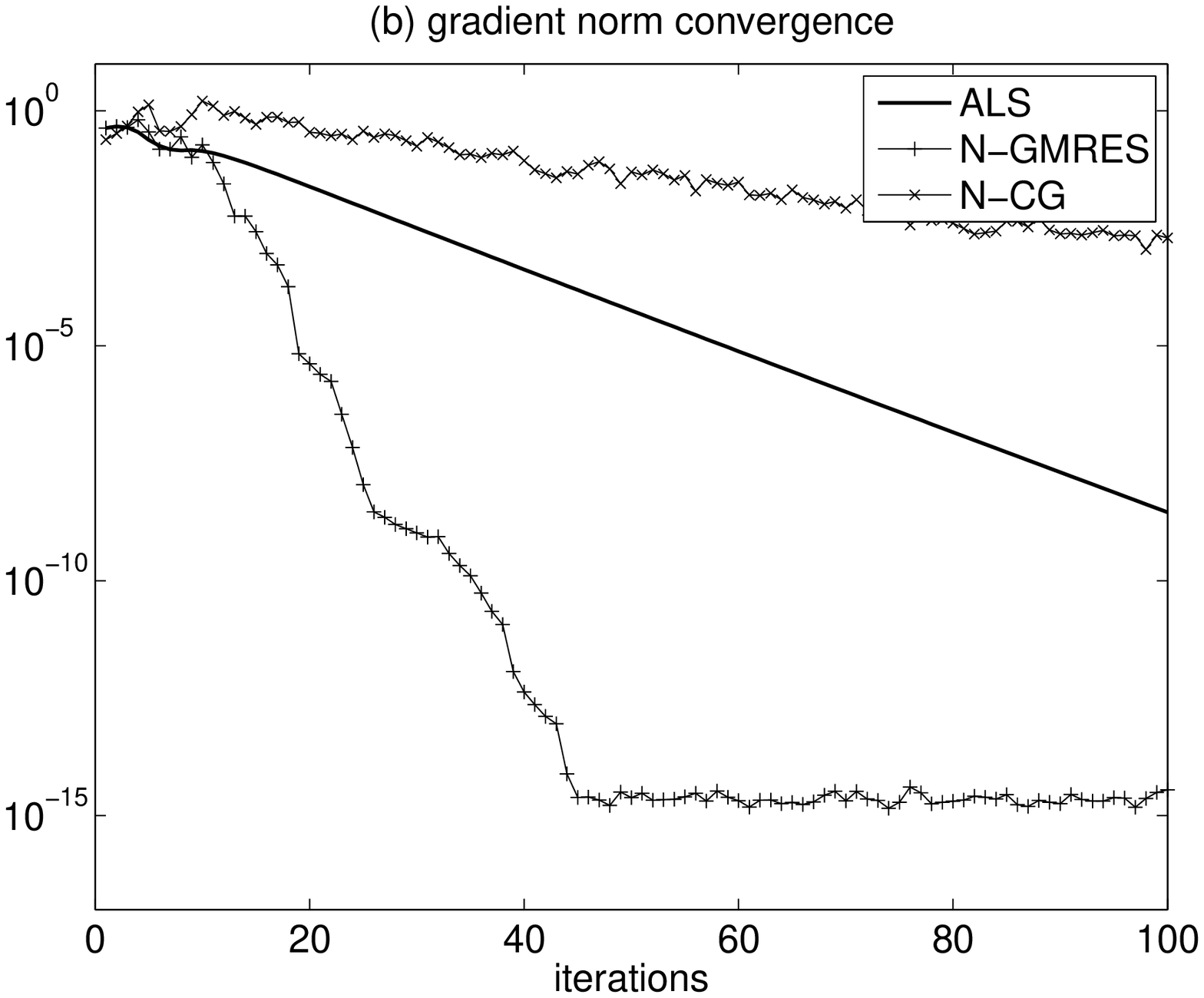}
  }
  \scalebox{0.35}{
  \includegraphics{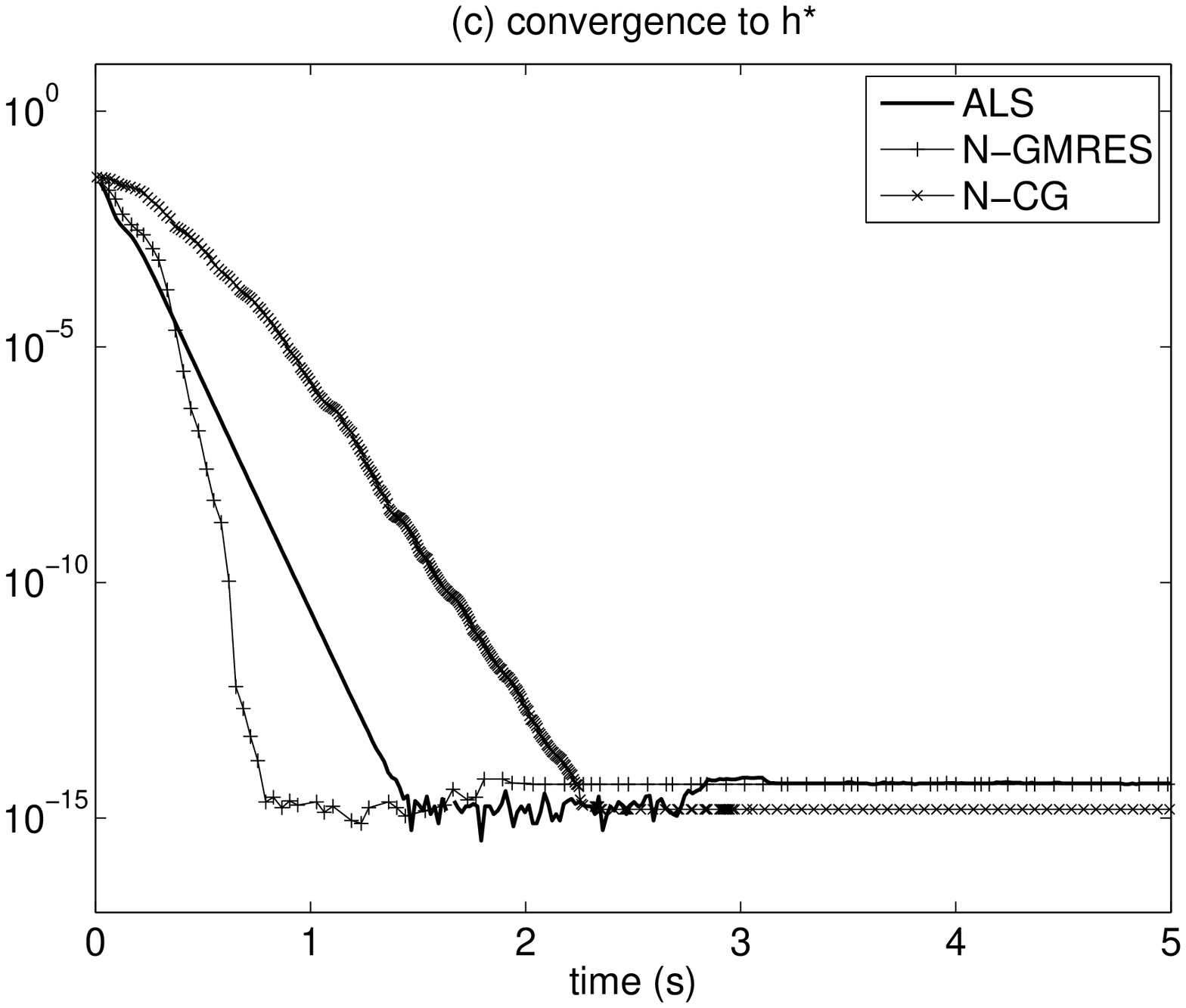}
  \includegraphics{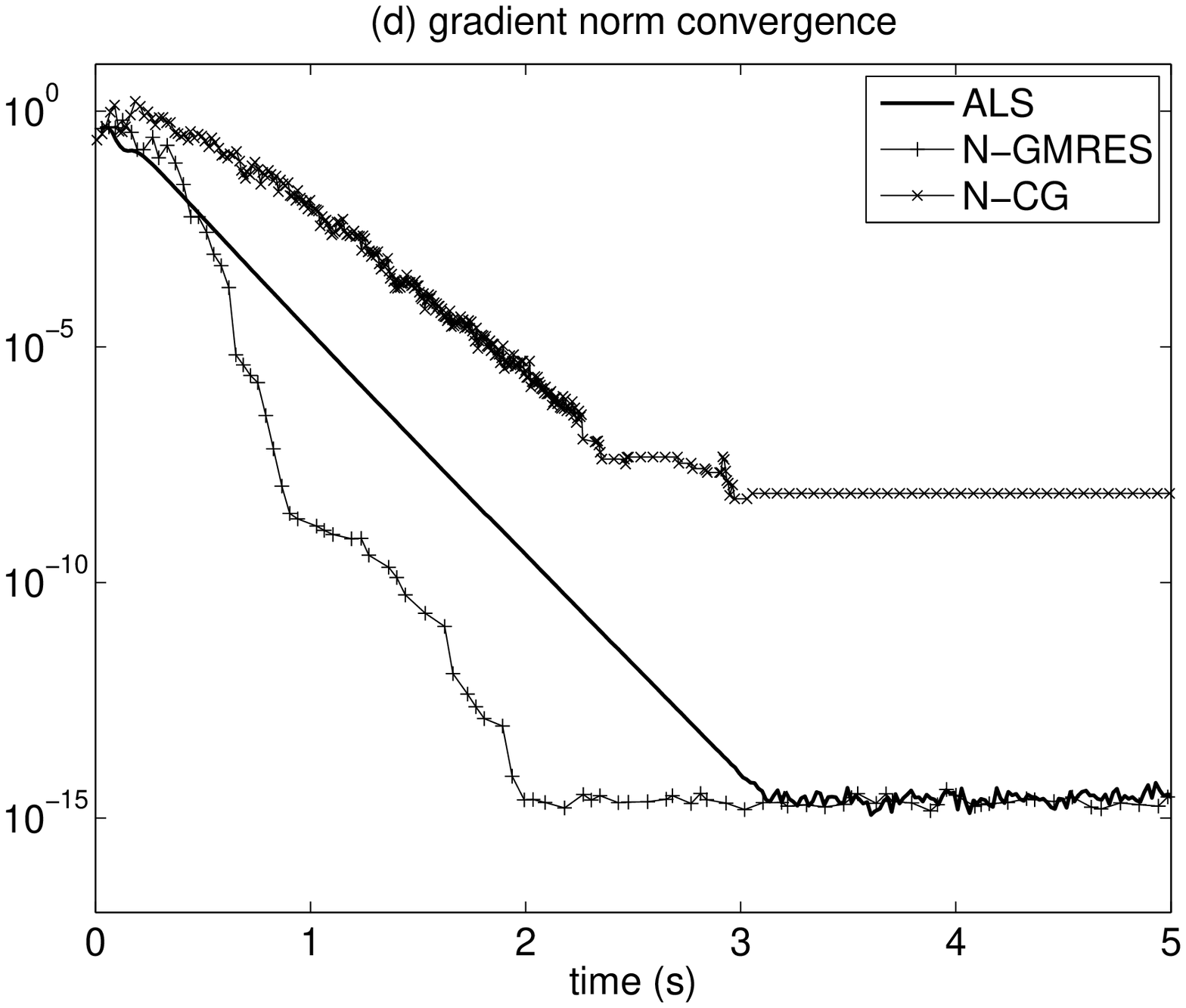}
  }
   \caption{Test Problem II. Convergence plots for case 4 from Tables \ref{tab:sparse3}-\ref{tab:sparse10}.}
   \label{fig:sparseCase5}
\end{figure}    

In this section, we present numerical results for a simple sparse test problem:\\

\noindent
{\sc Test Problem II:} Standard finite difference Laplacian tensor on a regular grid of size $s^d$ in $d$ dimensions. This test problem results in an $N$-way sparse data tensor $\mT$ with $N=2\,d$ and nonzero elements of value $2\,d$ and $-1$. For example, for $d=2$, $\mT \in \mathbb{R}^{s \times s \times s \times s}$, and the nonzero tensor elements are $t(i,j,i,j)=2 \,d=4$, for $i=1,\ldots,s$ and $j=1,\ldots,s$, and
$t(i,j,i+1,j)=-1$, $t(i,j,i,j+1)=-1$, $t(i,j,i-1,j)=-1$, and $t(i,j,i,j-1)=-1$ (with the usual exceptions at boundary points of the $d$-dimensional regular grid).

This is an academic test problem, but it provides sparse tensors that are suitable for testing our numerical method.

As in the previous section, we first consider the effect of varying the N-GMRES window size.
Fig.~\ref{fig:sparseDetIterations} shows convergence plots for an instance of Test Problem II with parameters $d=3$ and $s=6$.
Tensor $\mT$ is a 6-way tensor of size $6 \times 6 \times 6 \times 6 \times 6 \times 6$, and we seek a CP decomposition with $R=3$. Fig.~\ref{fig:sparseDetIterations}, with N-GMRES window size $w=20$, shows that ALS is rather slow for this problem. N-GMRES significantly reduces the number of iterations, faster than N-CG.

Fig.~\ref{fig:sparseDetTiming} shows convergence plots as a function of time for this test problem, for varying window size $w$. For this test problem, as in the previous section, window size $w=20$ also appears a suitable choice, and we use it for the remaining numerical tests in this section.

Tables \ref{tab:sparse3}-\ref{tab:sparse10} show convergence results for a series of sparse Test Problem II runs with a wide range of parameter values $N, s$ and $R$. We have found that the convergence behavior has more variation for Test Problem II than for Test Problem I. In particular, it happens more often that ALS, N-GMRES and N-CG converge to stationary points with different values of $h$, and the convergence profiles appear more sensitive to the initial guess. For this reason, we present two runs with different initial conditions for each parameter choice in Tables \ref{tab:sparse3}-\ref{tab:sparse10}, which give a more representative idea of how the methods perform on average. In the tables, we only present cases for which the three methods converge to a stationary point with the same value of $h$ (which still happens commonly).

Tables \ref{tab:sparse3}-\ref{tab:sparse10} largely confirm the observations we made for Test Problem I in the previous section.
Table \ref{tab:sparse3} shows that ALS is the fastest method for low accuracy computations.
Table \ref{tab:sparse6} shows that N-GMRES and N-CG become competitive for medium-accuracy computations, and Table \ref{tab:sparse10} shows that N-GMRES is almost always faster than ALS for high accuracy, sometimes substantially. N-GMRES is generally also faster than N-CG.

It is again instructive to consider convergence histories in some more detail for some of the test problems in Tables  \ref{tab:sparse3}-\ref{tab:sparse10}. Fig.~\ref{fig:sparseCase4} shows convergence plots for case 4 from Tables \ref{tab:sparse3}-\ref{tab:sparse10}. In this case, N-CG is fastest for medium accuracy, and N-GMRES for high accuracy. ALS converges rather slowly. Fig.~\ref{fig:sparseCase5} shows convergence plots for case 5 from Tables \ref{tab:sparse3}-\ref{tab:sparse10}.
In this case, ALS converges fast, but N-GMRES acceleration is still able to improve its convergence speeds for medium and high accuracy. We see again that the gradient convergence of N-CG stalls around $10^{-7}$, but $h$ does converge to $h^*$ up to machine accuracy in this case.


\section{Conclusion}
\label{sec:conc}
We have presented a new optimization algorithm for computing a canonical rank-$R$ tensor approximation that has minimal distance to a given tensor in the Frobenius norm.
The optimization algorithm uses a nonlinear version of GMRES iterate recombination that was proposed by Washio and Oosterlee for systems of nonlinear PDEs \cite{WashioNGMRES-ETNA}. We apply this nonlinear GMRES acceleration to the ALS method, with the goal of efficiently driving the gradient of the CP objective function to zero, and combine it with a line search for globalization. The resulting 3-step N-GMRES optimization algorithm can be interpreted as an acceleration process of a one-step stand-alone method, or, alternatively, the stand-alone method can be considered as a preconditioning process for N-GMRES. We have explained how N-GMRES can be applied to the (approximate) canonical tensor decomposition problem.

Extensive numerical tests on dense and sparse tensors with varying sizes and dimensions (up to 8) show that N-GMRES with ALS preconditioning in many cases outperforms pure ALS when high accuracy is required, while ALS remains faster for low accuracy and `easy' problems. N-GMRES is also competitive with the N-CG method studied in \cite{AcarCPOPT}, and appears to outperform it significantly in some cases. In cases where ALS-preconditioned N-GMRES is slower than pure ALS, it is rarely so by much, and in the other cases the potential speed gain is very substantial. For this reason, it may be a good strategy to put an N-GMRES acceleration wrapper around ALS if one does not know in advance whether ALS would converge slowly (which may depend on the problem and on the initial guess). Generally speaking, one may expect that adding an N-GMRES acceleration wrapper makes ALS convergence more robust.

It is a question of extensive current interest how Krylov methods can be generalized to tensor computations, see, for example, \cite{Savas}. Our work presents the application of a nonlinear Krylov-type method to a tensor optimization problem, and we obtain acceleration that is significant in many cases. Our approach uses a nonlinear generalization of Krylov techniques, and it is perhaps not surprising that this seems to be a promising approach for tensor minimization problems, which are indeed inherently nonlinear (more specifically, multilinear).

A promising avenue for further research is to explore how the proposed N-GMRES optimization algorithm can be used to accelerate existing methods other than ALS for canonical tensor decomposition. Similarly, it would also be interesting to investigate how the N-GMRES optimization algorithm can accelerate ALS-type algorithms (or other algorithms) for other tensor approximation problems, for example, the best rank-$(R_1,R_2,R_3)$ approximation of a tensor, see, e.g., \cite{Ishteva}.
Furthermore, the nonlinear GMRES optimization algorithm proposed in this paper 
is based on general concepts and may be applied to other nonlinear optimization problems.



\end{document}